\documentclass[11pt]{amsart}
\usepackage{marginnote}
\usepackage{latexsym, amssymb, url, mathrsfs, amsfonts, MnSymbol, amsmath}
\usepackage[all,cmtip]{xy}
\usepackage[pdftex,lmargin=1.4in, rmargin=1.4in,tmargin=1.25in,bmargin=1.25in, marginpar=2.9cm]{geometry}
\usepackage{color}
\DeclareMathAlphabet{\mathscr}{OT1}{pzc}{m}{it} 

\makeindex

\newcommand{\projection}{\mathrm{pr}}

\newcommand{\tth}{\tilde{\Theta}}

\newcommand{\tdelta}{\tilde{\delta}}
\newcommand{\tlambda}{\tilde{\lambda}}
\newcommand{\uth}{\underline{\Theta}}
\newcommand{\uuo}{\underline{\uo}}
\newcommand{\ttau}{{\bar{\tau}}}
\newcommand{\tauco}{{\tau_\co}}

\newcommand{\shimuradatum}{\mathfrak{D}}

\newcommand{\Sh}{\mathcal{X}}

\newcommand{\Shp}{X}

\newcommand{\shmuord}{{\mathcal S}}

\newcommand{\shpmuord}{S}

\newcommand{\shp}{\Shp} 

\newcommand{\T}{\mathcal{T}}

\newcommand{\CT}{\T}

\newcommand{\diffop}{\mathcal{D}}

\newcommand{\CD}{\mathcal{D}}

\newcommand{\ci}{C^{\infty}}

\newcommand{\shimuraop}{D}

\newcommand{\hdrone}{H^1_{\mathrm{dR}}}

\newcommand{\hdrtau}{H_{\mathrm{dR}, \tau}}

\newcommand{\hdrtauu}{H_{\mathrm{dR}, \tau^\ast}}

\newcommand{\tauu}{{\tau^\ast}}

\newcommand{\uo}{\omega}

\newcommand{\CE}{\mathcal{E}}

\newcommand{\gr}{\mathrm{gr}}

\newcommand{\rk}{\mathrm{rk}}

\newcommand{\cf}{\mathfrak{f}}

\newcommand{\cp}{\alpha}

\newcommand{\CA}{\mathcal{A}}

\newcommand{\Auniv}{\CA}

\newcommand{\dual}{t}

\newcommand{\Frob}{\mathrm{Frob}}

\newcommand{\be}{\mathbf{e}}

\newcommand{\BT}{Barsotti--Tate }

\newcommand{\HA}{{E}}

\newcommand{\ha}{{\rm ha}}

\newcommand{\cmfield}{F}

\newcommand{\realfield}{\cmfield_0}

\newcommand{\reflexfield}{F(\shimuradatum)}

\newcommand{\randomfield}{L}

\newcommand{\F}{{\mathbb F}}

\newcommand{\Fbar}{\overline{\F}}

\newcommand{\IC}{\mathbb{C}}

\newcommand{\IR}{\mathbb{R}}

\newcommand{\IQ}{\mathbb{Q}}

\newcommand{\Qbar}{\overline{\IQ}}

\newcommand{\CO}{\mathcal{O}}

\newcommand{\isomto}{\overset{\sim}{\rightarrow}}

\newcommand{\Tr}{\mathrm{Tr}}

\newcommand{\Witt}{{\mathbb W}}

\newcommand{\compact}{K}

\newcommand{\Heckealgebra}{\mathcal{H}}

\newcommand{\Satake}{\mathcal{S}}

\newcommand{\Res}{\mathrm{Res}}

\newcommand{\Gal}{\mathrm{Gal}}

\newcommand{\lattice}{\mathscr{L}}

\newcommand{\KS}{\mathrm{KS}}

\newcommand{\ks}{\mathrm{ks}}

\newcommand{\id}{\mathrm{id}}

\newcommand{\Para}{P_\mu}

\newcommand{\Levi}{J}

\newcommand{\Levii}{J_\mu}

\newcommand{\Uni}{U_\mu}

\newcommand{\Borel}{B}

\newcommand{\Borell}{B_\mu}

\newcommand{\Nilp}{N}

\newcommand{\Nilpp}{N_\mu}

\newcommand{\Torus}{T}

\newcommand{\Toruss}{T_\mu}

\newcommand{\GL}{\mathrm{GL}}

\newcommand{\adeles}{\mathbb{A}}

\newcommand{\ZZ}{\mathbb{Z}}

\newcommand{\schur}{\mathbb{S}}

\newcommand{\lcm}{\mathrm{lcm}}

\newcommand{\Sym}{\mathrm{Sym}}

\newcommand{\Isom}{\mathrm{Isom}}

\newcommand{\Hom}{\mathrm{Hom}}

\newcommand{\End}{\mathrm{End}}

\newcommand{\Gm}{\mathbb{G}_m}

\newcommand{\G}{\mathcal{G}}

\newcommand{\diag}{\mathrm{diag}}

\newcommand{\Spec}{\mathrm{Spec}}

\newcommand{\orbit}{\mathfrak{o}}

\newcommand{\co}{\mathfrak{o}}

\newcommand{\GSp}{\mathrm{GSp}}

\newcommand{\CG}{{\mathcal{G}}}

\newcommand{\pe}{{(p^e)}}

\newcommand{\pee}{{(p^\be)}}

\renewcommand{\MR}[1]{ }

\renewcommand{\mod}{\mathrm{mod}\;}
\renewcommand{\bmod}{\mod}

\theoremstyle{plain}

\newtheorem{innercustomthm}{Theorem}
\newenvironment{customthm}[1]
  {\renewcommand\theinnercustomthm{#1}\innercustomthm}
  {\endinnercustomthm}

\newtheorem{thm}{Theorem}
\newtheorem{conj}[thm]{Conjecture}
\numberwithin{thm}{subsection}
\newtheorem{cor}[thm]{Corollary}
\newtheorem{coro}[thm]{Corollary}
\newtheorem{lem}[thm]{Lemma}
\newtheorem{prop}[thm]{Proposition}

\theoremstyle{definition}

\newtheorem{defi}[thm]{Definition}

\theoremstyle{remark}
\newtheorem{remark}[thm]{Remark}
\newtheorem{rmk}[thm]{Remark}

\title[Entire theta operators at unramified primes]{Entire theta operators at unramified primes}

\date{\today}

\author[E.\  Eischen]{E.\ Eischen$^*$}\thanks{$^*$Partially supported by NSF Grants DMS-1559609 and DMS-1751281.}
\author[E.\ Mantovan]{E.\ Mantovan}

\address{E. E. Eischen\\
Department of Mathematics\\
University of Oregon\\
Fenton Hall\\
Eugene, OR 97403-1222\\
USA}
\email{eeischen@uoregon.edu}

\address{E. Mantovan\\
Department of Mathematics\\
Caltech\\
1200 E California Blvd\\
Pasadena, CA 91125\\
USA
}
\email{mantovan@caltech.edu}

\begin{document}

\bibliographystyle{amsalpha}  
\begin{abstract}
Starting with work of Serre, Katz, and Swinnerton-Dyer, theta operators have played a key role in the study of $p$-adic and $\bmod p$ modular forms and Galois representations.  
This paper achieves two main results for theta operators on automorphic forms on PEL-type Shimura varieties: 1) the analytic continuation at unramified primes $p$ to the whole Shimura variety of the $\bmod p$ reduction of $p$-adic Maass--Shimura operators {\it a priori} defined only over the $\mu$-ordinary locus, and 2) the construction of new $\bmod p$ theta operators that do not arise as the $\bmod p$ reduction of Maass--Shimura operators.  
While the main accomplishments of this paper concern the geometry of Shimura varieties and consequences for differential operators, we conclude with applications to Galois representations.  Our approach involves a careful analysis of the behavior of Shimura varieties and enables us to obtain more general results than allowed by prior techniques, including for arbitrary signature, vector weights, and unramified primes in CM fields of arbitrary degree.
\end{abstract}

\maketitle
\setcounter{tocdepth}{1}
\tableofcontents
\vspace{-0.25in}

\section{Introduction}

Starting with work of Serre, Swinnerton-Dyer, and Katz, theta operators have played key roles in the study of $p$-adic and $\bmod p$ modular forms and associated arithmetic data at a prime number $p$.  For example, the operator $\theta$ from \cite{serre, swinnerton-dyer} that acts on the $q$-expansion $f(q)$ of a modular form $f$ by $q\frac{d}{dq}$ is employed in constructions of $p$-adic $L$-functions in characteristic $0$, as well as in the proof of the weight part of Serre's conjecture in characteristic $p$.

Recently, the potential for theta operators to be similarly powerful in higher rank applications has led to their study in the setting of automorphic forms on Shimura varieties $\Sh$ of PEL type.  In characteristic $0$, $p$-adic theta operators arise as $p$-adic Maass--Shimura operators, i.e. as differential operators constructed from the Gauss--Manin connection and Kodaira--Spencer morphism analogously to the $\ci$ Maass--Shimura operators from, e.g., \cite{kaCM, hasv, ha86, shar, padiffops2, EDiffOps, zheng}.  They are defined on automorphic forms over the $\mu$-ordinary locus of $\Sh$, and there is a mathematical obstruction to analytically continuing them to the whole Shimura variety (as explained in \cite[Section 1.3]{EFGMM}).

On the other hand, as this paper illustrates, the $\mod p$
setting is fundamentally different, in the sense that theta operators are entire, i.e. can be analytically continued to the whole $\mod p$ Shimura variety, and furthermore,
there are more theta operators than just those arising as $\mod p$ reductions of Maass--Shimura operators.   In particular, 
by building on ideas introduced by Katz,
we obtain the following results when $p$ is unramified in the reflex field of $\Sh$.

\begin{customthm}{I}[Rough version of Theorems \ref{ANAlambda_thm} and \ref{ANA_thm}: Analytic continuation]\label{ThmI}
Reductions $\mod p$ of $p$-adic Maass--Shimura differential operators $\diffop^\lambda$, {\it } a priori defined only over the $\mu$-ordinary locus (where they raise the weight of an automorphic form by a weight $\lambda$), can be analytically continued to the entire $\mod p$ Shimura variety $\Shp$.  

More precisely, for good weights (as in Definition \ref{goodweight}), there is a differential operator $\Theta^\lambda$ defined on automorphic forms on $\Shp$ whose restriction to the $\mu$-ordinary locus coincides with $\HA^{|\lambda|/2}\cdot \diffop^\lambda$, where $\HA$ denotes the $\mu$-ordinary Hasse invariant.
\end{customthm}

By Proposition \ref{prop-sym}, the amount $\lambda$ by which Maass--Shimura operators can raise the weight of an automorphic form is always symmetric (in the sense of Section \ref{weights_sec}).  For applications to the weight part of Serre's conjecture, though, one would also like more control over the weights.  So the theta operators described in Theorem \ref{ThmII} below are a boon, since they also allow the weights to vary by certain non-symmetric amounts.  {\bf This new phenomenon only occurs when the ordinary locus is empty (i.e. when the prime $p$ is not totally split in the reflex field of $\Sh$).  This is the complement of the set of cases handled by \cite{EFGMM} and is specific to the $\mu$-ordinary setting for unitary Shimura varieties.}

\begin{customthm}{II}[Rough version of Theorem \ref{newops-thm}: New theta operators]\label{ThmII}
Assume $p$ does not split completely in the reflex field.  Then the class of $\mod p$ theta operators is larger than the class of $\mod p$ reductions of Maass--Shimura operators.

In particular, there are entire theta operators that raise the weights of $\mod p$ automorphic forms by different amounts from those allowed by the $\mod p$ reduction of $p$-adic theta operators.  
More precisely, if $\Upsilon$ is as in Equation \eqref{upsilon_def} and $\lambda$ is symmetric and simple (as in Definition \ref{simple_defi}), there is an entire differential operator 
$\tth^\lambda$ that
raises the weight of $\mod p$ automorphic forms on $\Shp$ 
of good weight by the non-symmetric weight
$(|\lambda|/2)\kappa_{\ha}+\tlambda^\Upsilon$, 
where $\kappa_{\ha}$ is the weight of the $\mu$-ordinary Hasse invariant and $\tlambda^\Upsilon$ is as in Definition \ref{twistweight_defi}.
\end{customthm}

Most of the work in this paper concerns the development of techniques to prove Theorems \ref{ThmI} and \ref{ThmII}.  Keeping in mind a key source of motivation for studying $\mod p$ theta operators in the first place, though, we conclude the paper by also addressing some effects of these operators on Galois representations.

\begin{rmk}\label{vary-rmk}
In the precise versions of these theorems later in this paper, we have finer control over the weight than these rough versions might suggest.  In particular, we can vary the weights at places corresponding to different primes dividing $p$, but for clarity of notation in this introduction, we have suppressed the corresponding subscripts and partial Hasse invariants.  Such control is important for anticipated applications to theta cycles in studying the weight part of Serre's conjectures and, as discussed in Section \ref{innovations-section} below, cannot be achieved via prior approaches.
\end{rmk} 

\subsection{Principal innovations and relationships with prior developments}\label{innovations-section}
Thanks to the approach developed in the present paper, which relies on the development of a theory of automorphic forms over the $\mu$-ordinary locus by the authors in \cite{EiMa} and the construction of $\mu$-ordinary Hasse invariants by Goldring and Nicole in \cite{GN}, Theorems \ref{ThmI} and \ref{ThmII} improve on the previous results for theta operators in the setting of unitary groups of higher rank.
(See Remark \ref{earlier-results} below, for references to earlier work in the Hilbert--Siegel case.) 
A careful analysis of intrinsic properties of the underlying Shimura varieties enables us to remove restrictions from prior results. 

Building on Katz's study of the theta operator for modular forms in \cite{Katz-theta}, Theorem \ref{ThmI} was previously proved jointly by the authors together with Flander, Ghitza, and McAndrew in \cite[Theorem A]{EFGMM} under the auxiliary assumption that $p$ splits completely in the reflex field of $\Sh$ (i.e. when the ordinary locus is nonempty).  When the ordinary locus is empty, the constructions in \cite{EFGMM} still hold but the resulting $\Theta$-operators vanish on the whole $\mod p$ Shimura variety $\Shp$.  Theorem \ref{ThmII} was previously proved by de Shalit and Goren in  \cite{DSG, DSG2} in the special case of scalar-valued automorphic forms and under the assumption that
the real field associated to the Shimura datum defining $\Sh$ is $\IQ$.    (Note that the operator they denote by $\Theta$ is the operator we denote by $\underline{\Theta}$.)  The approach in \cite{DSG, DSG2}, which is completely different from that in this paper (as, for example, their proof relies on Fourier--Jacobi expansions), does not readily extend to the case of non-scalar weights.  That is, extending their techniques to non-scalar weights is not merely a notational or combinatorial issue.

In the present paper, like in \cite{EFGMM}, our approach to studying theta operators is coordinate-free and avoids $q$-expansions, Fourier--Jacobi expansions, Taylor series, Serre--Tate expansions, etc, and instead relies on intrinsic geometric features of the underlying Shimura variety.  (Even though other references also construct theta operators without referencing such expansions, their proofs of results about them sometimes rely in key ways on such expansions, for example to obtain stronger results under particular conditions, like discussed below.)  This allows us to continue to work with vector weights like in \cite{EFGMM}.

A key ingredient for extending the approach of \cite{EFGMM} to the case where $p$ is merely unramified in the reflex field (i.e. the ordinary locus need not be nonempty) is the $\mu$-ordinary Hasse invariants introduced by Goldring and Nicole in \cite{GN}.  Working with partial Hasse invariants, in place of the Hasse invariant from the case of nonempty ordinary locus, enables us to extend the Hodge--de Rham splitting in characteristic $p$ to the whole Shimura variety in a way that enables us to naturally extend the mod $p$ reduction of Maass--Shimura operators to the whole Shimura variety.  This method has the advantage that it allows us to vary weights by different amounts at places corresponding to different primes dividing $p$, as mentioned in Remark \ref{vary-rmk}, but also the disadvantage that forces us to restrict to good weights in Theorems \ref{ThmI} and \ref{ThmII}.

As a crucial intermediate step introduced in the present paper, we also consider 
differential operators on the {\it OMOL} (``Over the Mu-Ordinary Locus'') sheaves introduced in \cite[Section 4.2]{EiMa}, when the ordinary locus is empty (i.e. when $p$ does not split completely in the reflex field).  In this paper, we establish the analytic continuation of the $\mod p$ reduction of OMOL sheaves and differential operators to the whole $\mod p$ Shimura variety (Theorem \ref{OMOLdiffextend}). 
As explained in \cite{EiMa},  over the $p$-adic $\mu$-ordinary locus, there is a canonical projection from an automorphic sheaf to an OMOL sheaf of the same weight, which is generally not an isomorphism.   In the present paper, we observe that, over the whole $\mod p$ Shimura variety, we also have an injection from an OMOL sheaf of simple weight (as in Definition \ref{simple_defi}) to an automorphic sheaf, of good higher weight.  This enables us to exploit  the analytic continuation
of  OMOL sheaves and differential operators to construct new entire theta operators,  when $p$ does not split completely in the reflex field.

Theorem \ref{ThmII} is then achieved by exploiting the results we develop for OMOL sheaves and their relations to automorphic sheaves in the $\mod p$ setting. 
As noted in Remark \ref{DSG_compare}, in the special case when the automorphic forms have scalar weights and the base field is $\IQ$, Theorem \ref{ThmII} is  also proved  in \cite{DSG} and \cite[Sections 4 and 5]{DSG2}, although 
the operator $\Theta$ in  {\it loc. cit.} can only be iterated  (to define operators $\Theta^\lambda$, for higher weights $\lambda$) when the signature is $(n,1)$ (as opposed to general signature).   The scalar weights occurring in {\it loc. cit.} are special cases of what we call {\it simple, scalar} in the present paper.  

The techniques we use in our proofs (e.g. exploiting the underlying geometry and OMOL sheaves, and avoiding $q$-expansions) are genuinely different from those in \cite{DSG,DSG2} and immediately remove their conditions on the signature.  (To be clear, the construction of theta operators in either case does not require $q$-expansions, but rather, the difference is in the techniques employed in proofs.)  On the other hand, in \cite{DSG, DSG2}, de Shalit and Goren obtain an operator $\Theta$ which raises the weight by a lower amount,
via a lower exponent on the Hasse invariant than our methods produce in the cases they consider.  Their better bound on the weight is useful, for example, in their application to the study of theta cycles at signature $(n,1)$ in \cite[Section 5]{DSG}.

As a consequence of our work, the results on Hecke algebras and Galois representations from \cite[Sections 4 and 5]{EFGMM} are extended in Section \ref{GalOne-section} to the case where $p$ need not split (but rather is merely unramified) and where the set of weights under consideration is expanded.

\begin{rmk}\label{earlier-results}
While the discussion in the present paper focuses primarily on 
overcoming challenges and exploring new phenomena particular to the $\mu$-ordinary setting (that is, the case when the ordinary locus is empty, which is specific to the unitary case),
we note for the sake of completion that in the special case of the symplectic group $\GSp_4\left(\IQ\right)$, Yamauchi has produced precise results for theta cycles, which rely on combinatorics specific to that case \cite{Yama}.  Results for theta operators in the Hilbert--Siegel case of arbitrary rank are also obtained in \cite{EFGMM}.  Earlier, Andreatta and Goren also produced stronger results on theta operators and theta cycles in the setting of Hilbert modular forms, i.e. for $\GSp_2 = \GL_2$ over a totally real field \cite[Section 16]{AndreattaGoren}.  Those results build on Katz's weight filtration theorem (which Jochnowitz and Edixhoven had also used earlier to obtain results about theta cycles in the setting of modular forms \cite{jochnowitz, Edixhoven}).
\end{rmk}

\subsection{Structure of the paper}
Sections \ref{background-section} and \ref{DiffOpsReview-section} discuss properties of the main objects with which we work.  In particular, Section \ref{background-section} introduces background information for Shimura varieties, automorphic forms, and partial Hasse invariants over the $\mu$-ordinary locus.  After recalling constructions of $p$-adic Maass--Shimura operators over the $\mu$-ordinary locus \cite{EDiffOps, EFMV, EiMa}, Section \ref{DiffOpsReview-section} establishes key properties of these operators.  Section \ref{DiffOpsReview-section} concludes with results for differential operators on {\it OMOL} automorphic forms, which were first introduced in \cite[Section 4.2]{EiMa} and play a crucial intermediate role in achieving the results of the present paper.

Our approach to constructing entire $\mod p$ theta operators relies heavily on Section \ref{HD-section}, which concerns the interplay between the characteristic $p$ Hodge--de Rham filtration and partial Hasse invariants (from \cite{GN}).  In particular, Theorem \ref{PEP_thm} is a key ingredient for extending the analytic continuation results from \cite[Section 3]{EFGMM} (i.e. when the prime $p$ splits completely, so the ordinary locus is nonempty) to the setting of Theorem \ref{ThmI} ($p$ merely unramified, so the ordinary locus need not be nonempty).  Employing Theorem \ref{PEP_thm}, Section \ref{AC-section} details how to extend the mod $p$ reduction of $p$-adic Maass--Shimura varieties, initially defined only over the $\mu$-ordinary locus, to the entire Shimura variety.   

While Section \ref{AC-section} concerns the mod $p$ reduction of $p$-adic differential operators, Section \ref{modpdiff_sec} produces the new classes of mod $p$ differential operators arising in Theorem \ref{ThmII}.  These new operators raise the weights of automorphic forms by amounts different from the amounts possible with the mod $p$ reductions of Maass--Shimura operators.  As an intermediate step, Section \ref{modpdiff_sec} also explains how to analytically continue the mod $p$ reduction of  differential operators on {\it OMOL} $p$-adic automorphic forms.  We anticipate that the additional control over the weights will be useful for studying theta cycles and Serre's weight conjecture.  Motivated by this anticipated application, we apply our mod $p$ differential operators to Galois representations in Section \ref{GalOne-section}, via an analysis of their interaction with Hecke operators.

\subsection{Acknowledgements}

Our work on this project has benefitted from helpful conversations with Ehud de Shalit, Alex Ghitza, Eyal Goren, and Angus McAndrew, especially about their earlier papers on related topics.  We completed key steps during our visit to the University of Lille, as well as the second named author's visits to the University of Oregon and the University of Padua, and we are grateful to these institutions for their hospitality.  We are also grateful to the referee for helpful suggestions.

\section{Background and setup}\label{background-section}

In the section, we introduce notation, key assumptions, and basic information about automorphic forms on PEL type Shimura varieties (Sections \ref{Sh_sec} and \ref{auto_sec}), the Hasse invariant and $p$-adic automorphic forms over the $\mu$-ordinary locus (Sections \ref{hasselocus_sec} and \ref{muord_sec}), and the associated Hecke algebras and Galois representations (Section \ref{galois_sec}).

\subsection{Shimura data and Shimura varieties}\label{Sh_sec}
We briefly introduce PEL-type Shimura varieties of unitary (A) or symplectic (C) type.  For a more extensive introduction to Shimura varieties, see \cite{kottwitz, lan-examples, la}.  Given that many of the ingredients for our work exist for Shimura varieties of Hodge type, we expect it is possible to extend our results from unitary and symplectic groups to that more general setting.

To the extent reasonable, we employ the conventions of \cite[Section 2.2]{EFGMM}.  {\bf Note, however, that the simplifying conditions of \cite[Section 2.2.2]{EFGMM} (i.e. that $p$ splits completely in the reflex field) are not imposed here, since one of the achievements of the present paper is their removal.}  
\subsubsection{Shimura data and associated data}
Our Shimura varieties are associated to a PEL-type Shimura datum, i.e. a tuple $\shimuradatum:=\left(D, \ast, V, \langle, \rangle, h\right)$\index{$\shimuradatum$} consisting of:
\begin{itemize}
\item{A finite-dimensional simple $\IQ$-algebra $D$}
\item{A positive involution $\ast$ on $D$ over $\IQ$}
\item{A nonzero finitely-generated left $D$-module $V$ together with a non-degenerate $\IQ$-valued alternating form $\langle, \rangle$ on $V$ such that $\langle bv, w\rangle = \langle v, b^\ast w\rangle$ for all $v, w\in V$ and $b\in D$}
\item{A $\ast$-homomorphism $h: \IC\rightarrow C_\IR$, where $C := \End_D(V)$ viewed as a $\IQ$-algebra and the symmetric real-valued bilinear form $\langle \cdot, h(i)\cdot\rangle$ on $V_\IR$ is positive-definite}
\end{itemize}
From the Shimura datum $\shimuradatum$, we also obtain:
\begin{itemize}
\item{A field \index{$\cmfield$}$\cmfield$, defined to be the center of $D$.}
 \item{A totally real field \index{$\realfield$}$\realfield$, defined to be the fixed field of $\ast$ on $\cmfield$.}
 \item{A decomposition $V_{\IC} = V_1\oplus V_2$ arising from the endomorphism $h_{\IC} = h\times_{\IR}\IC$ of $V_{\IC} = V_{\IR}\otimes_{\IR} \IC = V\otimes _{\IQ}\IC$ on which $(h(z), 1) = h(z)\times 1$ acts by $z$ on $V_1$ and by $\bar{z}$ on $V_2$ for each $z\in \IC$.}
 \item{An integer $n:=\dim_\cmfield V$.}
 \item{The reflex field\footnote{Standard notation for the reflex field, including in the authors' prior work, is $E$.  In this paper, though, we follow the convention of using $E$ for the Hasse invariant.} \index{$\reflexfield$}$\reflexfield$, defined to be the field of definition of the $G(\IC)$-conjugacy class of $V_1$.}
 \end{itemize}

 We have $\cmfield= \realfield$ if $\cmfield$ is totally real, and otherwise $\cmfield$ is a CM field obtained as an imaginary quadratic extension of the totally real field $\realfield$.  We fix an algebraic closure $\Qbar$ of $\IQ$.
 Given a number field $\randomfield$, we denote by $\T_\randomfield$\index{$\T_\randomfield$} the set of embeddings $\randomfield\hookrightarrow \Qbar$.  Since we will be working over both $\IC$ and $\IC_p$, we also fix embeddings
\begin{align*}
\iota_\infty&: \Qbar\hookrightarrow\IC\\
\iota_p&:\Qbar\hookrightarrow\IC_p,
\end{align*}
and we identify $\Qbar$ with its image under each of these embeddings.  So via composition with $\iota_\infty$, $\T_{\realfield}$ is the set of embeddings $\realfield\hookrightarrow\IR$.  If $\realfield\neq \cmfield$, the elements of $\T_{\cmfield}$ arise in complex conjugate pairs $\tau\neq\tau^\ast$\index{$\tau^\ast$} with $\tau|_{\realfield} = \tau^\ast|_{\realfield}$.  We denote by $\Sigma_{\cmfield}$ a choice of CM type, i.e. a set consisting of exactly one of $\tau$, $\tau^\ast$ for each complex conjugate pair $(\tau, \tau^\ast)$ of complex embeddings of $\cmfield$.  We identify $\Sigma_\cmfield$\index{$\Sigma_\cmfield$} with $\T_{\realfield}$ via the bijective map $\tau\mapsto \tau|_{\realfield}$.  We also sometimes drop the subscript and write $\T$\index{$\T$} when the subscript is clear from context.

\subsubsection{Additional conditions}\label{additional-conditions}
We assume the Shimura datum $\shimuradatum$ satisfies the following additional conditions:
\begin{itemize}
\item{The prime $(p)$\index{$(p)$} is unramified in $\cmfield$.}
\item{The algebra $D$ is split at $p$, i.e. $D_{\IQ_p}$ is a product of matrix algebras over extensions of $\IQ_p$.}
\item{There is a $\ZZ_{(p)}$-order $\CO_D$ in $D$ preserved by $\ast$ and whose $p$-adic completion is a maximal order in $D_{\IQ_p}$.}
\item{There is a $\ZZ_p$-lattice $\lattice\subset V_{\IQ_p}$ self-dual with respect to $\langle, \rangle$ and preserved by $\CO_D$.}
\end{itemize}
We fix a prime $\mathfrak{p}$\index{$\mathfrak{p}$} in $\reflexfield$ above $(p)$, and write \index{$k(\mathfrak{p})$}$k(\mathfrak{p}):=\CO_{\reflexfield}/\mathfrak{p}$ for its residue field. Under the above assumption, $\reflexfield$ is unramifed at $\mathfrak{p}$.  We fix an algebraic closure ${\mathbb F}$\index{${\mathbb F}$} of $k(\mathfrak{p})$, define 
$\Witt:=W\left({\mathbb F}\right)$\index{$\Witt$} its ring of Witt vectors,
and write $\sigma$\index{$\sigma$} for the absolute Frobenius on $\Witt$.

In the following, for any number field $\randomfield$, we denote by $\randomfield^\Gal$ its Galois closure inside $\Qbar$. For any field $k$ of characteristic $p$, we denote by $W(k)$ its ring of Witt vectors.   Given a field $\randomfield$, we denote by $\CO_\randomfield$ its ring of integers, and given a prime $\mathfrak{q}$ in $\randomfield$, we write $\CO_{\randomfield, \mathfrak{q}}$ (resp. $\CO_{\randomfield_\mathfrak{q}}$) for the localization (resp. completion) at $\mathfrak{q}$.   If $k$ is the residue field of a complete field $\randomfield$ which is unramified, we identify $W(k)$ with the ring of integers $\CO_\randomfield$.   
Given a number field $\randomfield$, and a prime $\mathfrak{q}$ above $(p)$, if $\mathfrak{q}$ is unramified,  we identify $\T_\randomfield$ with $\Hom\left(\CO_\randomfield, \Witt\right)$, via $\iota_p$. Composition on the right defines an action of Frobenius $\sigma$ on $\T_\randomfield$; for any $\tau\in\T_\randomfield$ its $\sigma$-orbit is the subset\index{$\co_\tau$}
\begin{align*}
\co_\tau=\{\tau\circ\sigma^i|i\in\ZZ\}\subseteq \CT_\randomfield.
\end{align*} 
We define\index{$\mathfrak{O}_\randomfield$}
\begin{align*}
\mathfrak{O}_\randomfield:=\left\{\co\subset \T_L\mid \co \mbox{ is a $\sigma$-orbit}\right\}.
\end{align*}

 \subsubsection{Algebraic group associated to the Shimura datum $\shimuradatum$}\label{groups-intro}
We denote by $G$ the algebraic group defined over $\IQ$ whose $R$-points, for any $\IQ$-algebra $R$, are
 \begin{align*}
G(R)=\left\{x\in C\otimes_{\IQ}R| xx^\ast\in R^\times\right\}.
\end{align*}
We denote by
\begin{align*}
\nu: G\rightarrow \Gm
\end{align*}
the similitude factor of $G$.
For any character $\psi: G\rightarrow \Gm$, we denote by $\hat\psi$ its cocharacter $\Gm\rightarrow \hat G$.  Note that $\widehat{(\nu^m)} = {(\hat\nu)}^m$ for each $m\in\ZZ$.  We define\index{$G_1$} 
\begin{align*}
G_1:=\ker(\nu).
\end{align*}  We have 
\begin{align*}
G_1=\Res_{F_0/\IQ}(G_0)
\end{align*}
 for some algebraic group $G_0$ defined over $F_0$.  If $\cmfield\neq \realfield$, then $G_0$ is an inner form of a quasi-split unitary group over $\realfield$.  In this case, our Shimura datum is of unitary type (A).  On the other hand, if $\cmfield=\realfield$, then over an algebraic closure of $\realfield$, $G_0$ is orthogonal (type D) or symplectic (type C).  In this paper, we focus on types A and C.

Let \index{$\compact$}$\compact=\compact^p\compact_p$, with \index{$\compact_p$}$\compact_p\subset G\left(\IQ_p\right)$ hyperspecial (i.e. $\compact_p$ is the stabilizer of $\lattice$) and $\compact^p\subset G\left(\adeles_f^p\right)$, be an open compact subgroup of $G\left(\adeles_f\right)$ that is neat in the sense of \cite[Definition 1.4.1.8]{la}.  (Following the usual conventions, $\adeles_f$ denotes the finite adeles of $\IQ$, and $\adeles_f^p$ denotes the finite adeles away from $p$.)  In other words, $\compact$ is a level.  Given a finite place $v$, we say that $v$ is {\it good} with respect to $K$ and $p$ if $v\ndivides p$ and $\compact_v$ is hyperspecial at $v$.  Otherwise, we say $v$ is {\it bad} with respect to $K$ and $p$.  We denote by $\Sigma_{\compact, p}$\index{$\Sigma_{\compact, p}$} the set of places that are bad with respect to $\compact$ and $p$.

\subsubsection{Moduli space of abelian varieties associated to the Shimura datum $\shimuradatum$}

Associated to our PEL-type Shimura datum and level $\compact$ is a moduli space $\Sh:=\Sh_\compact$\index{$\Sh$}\index{$\Sh_\compact$} parametrizing $\shimuradatum$-enriched abelian varieties, i.e. abelian varieties  together with polarization $A\rightarrow A^\dual$, endomorphism, and $\compact$-level structure, satisfying Kottwitz's determinant condition (we refer to \cite[p. 390]{kottwitz}, for details).  (In this paper, we use a superscript $\dual$ to denote the dual of an object.)  Under the conditions of Section \ref{additional-conditions}, $\Sh$ canonically extends to a smooth quasi-projective scheme over $\CO_\reflexfield\otimes_\ZZ\ZZ_{(p)},$ still denoted $\Sh$ or $\Sh_\compact$.  We regard $\Sh=\Sh_\compact$ as a scheme over $\CO_{\reflexfield, \mathfrak{p}}$.  We refer to $\Sh_\compact$ as {\it the PEL-type Shimura variety of level $K$} associated to our choice of Shimura datum, and denote by $\Shp$\index{$\Shp$} or $\Shp_\compact$\index{$\Shp_\compact$} its reduction modulo $\mathfrak{p}$, i.e. $\Shp:=\Sh\times_{\CO_{\reflexfield, \mathfrak{p}}} k(\mathfrak{p})$.
By abuse of notation, we also denote by $\Sh$, resp. $\Shp$, the schemes $\Sh\times_{\CO_{\reflexfield, \mathfrak{p}}} \Witt$, resp. $\Shp\times\mathbb{F}$.

\subsubsection{Decompositions and signatures associated to the Shimura datum}\label{signature-intro}
As noted in \cite[(2.0.3)]{kaCM}, for any $\CO_{\realfield^\Gal}$-algebra $R$, the ring homomorphism
\begin{align*}
\CO_{\realfield}\otimes_{\ZZ} R&\rightarrow \oplus_{\tau\in\T_{\realfield}}R\\
a\otimes r&\mapsto (\tau(a)r)_{\tau\in\T_{\realfield}}
\end{align*}
is an isomorphism if and only if the discriminant $d_{\realfield}$\index{$d_{\realfield}$} of $\realfield/\IQ$ is invertible in $R$.  Given $\tau\in\T_{\realfield}$ and an $\CO\otimes R$-module $M$, we denote by $M_\tau$ the submodule of $M$ annihilated by the set of $a\otimes 1-1\otimes \tau(a)\in \CO_{\realfield}\otimes R,$ i.e. the submodule on which each $a\in\CO_{\realfield}$ acts as scalar multiplication by $\tau(a)$.
If $R$ is a $\CO_{\realfield^\Gal}$-algebra in which $d_{\realfield}$ is invertible and $M$ is a locally free $\CO_{\realfield}\otimes R$-module, then similarly to \cite[(2.0.9)]{kaCM}, we have that the canonical $\CO_{\realfield}\otimes R$-module-homomorphism
\begin{align*}
M\rightarrow\oplus_{\tau\in\T_{\realfield}} M_\tau
\end{align*}
is an isomorphism.
If $\cmfield\neq\realfield$ and $R$ is, in addition, an $\CO_{\cmfield^\Gal}$-algebra, then the action of $\cmfield$ induces a further decomposition
\begin{align*}
M = \oplus _{\tau\in\T_{\cmfield}} M_\tau = \oplus_{\tau\in\Sigma_\cmfield}M_\tau\oplus M_{\tau^\ast} = \oplus_{\sigma\in\T_{\realfield}}M_\sigma^+\oplus M_\sigma^-,
\end{align*}
where, for each $\sigma\in\T_{\realfield},$ $M_\sigma^+$ (resp. $M_\sigma^-$) is the submodule of $M_\sigma$ on which each $a\in\CO_{\cmfield}$ acts as scalar multiplication by $\tau(a)$ (resp. $\tau^\ast(a)$) with $\tau$ the element of $\Sigma_\cmfield$ such that $\tau|_{\realfield} =\sigma$.

For $i=1, 2$, we have a decomposition
\begin{align*}
V_i = \oplus_{\tau\in\T_\cmfield}V_{i, \tau}
\end{align*}
induced by the decomposition $F\otimes_\IQ\IC = \oplus_{\tau\in\T_\cmfield}\IC$ (identifying $a\otimes b$ with $\left(\tau(a)b\right)_{\tau\in\T_\cmfield}$).  Thus we also have a decomposition
\begin{align*}
V_\IC = \oplus_{\tau\in\T_\cmfield}V_\tau,
\end{align*}
with
\begin{align*}
V_\tau:=V_{1, \tau}\oplus V_{2, \tau}
\end{align*}
for all $\tau\in\T_\cmfield$.

For each $\tau\in \T_\cmfield$, we set\index{$a_\tau$}
\begin{align*}
a_\tau := \dim_\IC V_{1, \tau}.
\end{align*}
The {\it signature} of the Shimura datum is $\left(a_\tau\right)_{\tau\in\T_\cmfield}$.  For each $\tau\in\Sigma_\cmfield$, we have
\begin{align*}
n=\begin{cases} a_\tau+a_{\tau^\ast}& \mbox{ in the unitary case (A)}\\
a_\tau & \mbox{ in the symplectic case (C)}.
\end{cases}
\end{align*}
In Case A (so $\cmfield\neq\realfield$), this is the signature of the unitary group $G_0/\realfield$ and the signature at $\tau\in\Sigma_\cmfield$ is $\left(a_\tau^+, a_\tau^-\right)$ with $a_\tau^+:=a_\tau$ and $a_\tau^-:=a_{\tau^\ast}$.  Following the conventions of \cite{moonen, EiMa}, we also define\index{$\cf(\tau)$}
\begin{align*}
\cf(\tau):=a_\tau
\end{align*}
for each $\tau\in\T_\cmfield$, and we denote by\index{$\cf$}
\begin{align*}
\cf := \left(\cf\left(\tau\right)\right)_\tau.
\end{align*}
the {\em signature} of the Shimura datum $\shimuradatum$.

 \subsection{Weights, representations, and automorphic forms} \label{auto_sec}
 We summarize key details about weights and representations, following the approaches of \cite[Sections 2.3 through 2.5]{EiMa} and \cite[Sections 2.1 and 2.2]{EFGMM}.  For additional details, see \cite[Sections 2.3 and 2.4]{EFMV} or \cite[Section 3.2]{CEFMV}.
 \subsubsection{Subgroups}\label{weights-subgroups}
 We denote by $\Levi$\index{$\Levi$} the algebraic group over $\ZZ$ defined by
 \begin{align*}
 \Levi:=\prod_{\tau\in\T_\cmfield}\GL_{a_\tau} = \begin{cases}\prod_{\tau\in\Sigma_\cmfield}\left(\GL_{a_\tau}\times\GL_{a_{\tau^\ast}}\right)\subseteq\prod_{\tau\in\Sigma_\cmfield}\GL_n=\prod_{\tau\in\T_{\realfield}}\GL_n, & \mbox{ in the unitary case (A)}\\
 \prod_{\tau\in\T_{\realfield}}\GL_n, & \mbox{ in the symplectic case (C)}\end{cases}
 \end{align*}
We denote by $\Borel$ a Borel subgroup of $\Levi$, $\Torus$\index{$\Torus$} a maximal torus contained in \index{$\Borel$}$\Borel$, and \index{$\Nilp$}$\Nilp$ the unipotent radical of $\Borel$.
We have a decomposition $\Torus = \prod_{\tau\in\T_\cmfield}\Torus_\tau$, and we have analogous decompositions, denoted analogously, for each algebraic subgroup of $\Levi$.  By choosing an ordered basis for $V_\IC$ compatible with the decompositions from Section \ref{signature-intro}, we identify $\Levi(\IC)$ with a Levi subgroup of $G_1(\IC)$.  We choose such a basis so that furthermore $\Borel_\tau$ is identified with the subgroup of upper triangular matrices in $\GL_{a_\tau^+}$
and $\Torus_\tau$ with $\Torus_{a_\tau^+}:=\Gm^{a_\tau^+}$, which is, in turn, identified with the subgroup of diagonal matrices of $\Levi$.  Note that each of these groups is split over $\CO_{\reflexfield_\mathfrak{p}}$, i.e. any maximal torus in it is isomorphic over $\CO_{\reflexfield_\mathfrak{p}}$ to a product of copies of $\Gm$.

 For a choice of an ordered partition $m_\bullet$ given by
\begin{align}\label{partition-equ}
m_{1, \tau}+\cdots + m_{s_\tau, \tau} = a_\tau,
\end{align}
for each $\tau\in \T_\cmfield$, 
we denote by $P=P_{m_\bullet}$ the associated block upper triangular parabolic subgroup  of $\Levi$ containing $\Borel$, and by $U=U_{m_\bullet}$  the unipotent radical of $P$. Then the Levi subgroup $M=M_{m_\bullet}$ of $P$, $M\cong P/U$, is a block diagonal product of groups 
$\GL_{m_{t, \tau}}$, $t=1, \ldots, s_\tau$, $\tau\in\T_\cmfield$.

Our choice of ordered partitions of $a_\tau$, for all $\tau\in\T_\cmfield$, will be uniquely determined by the geometry of the underlying Shimura variety and its $\mu$-ordinary locus (see 
Equation (\ref{ordered_partition}) below and \cite[\S 2.9]{EiMa} for a detailed explanation). In the following, for $m_\bullet^\mu$ the partition given in Equation (\ref{ordered_partition}), we write
 $\Para=P_{m_\bullet^\mu}$\index{$\Para$}, $\Uni=U_{m_\bullet^\mu}$\index{$\Uni$}, and 
  \index{$\Levii$}$\Levii=M_{m_\bullet^\mu}\cong \Para/\Uni$.  
  
  Note that for any representation $\rho$ of $\Levi$, the associated graded representation $\gr\left(\rho|_{\Para}\right)$ of $\Levii$ and $\rho|_{\Levii}$ are canonically identified.  We also define a Borel subgroup \index{$\Borell$}$\Borell=\Borel\cap\Levii$ of $\Levii$ and \index{$\Nilpp$}$\Nilpp=\Nilp\cap\Levii$ and \index{$\Toruss$}$\Toruss=\Torus\cap\Levii$ the unipotent radical and maximal torus of $\Borell$, respectively.

 \subsubsection{Weights}\label{weights_sec}

We briefly introduce weights and their relationship with algebraic representations.  For more details, see \cite[Section 2.3]{EiMa} \cite[Section 2.4]{EFMV}, \cite[Sections 5.1.3 and 8.1.2]{hida}, \cite[Part II, Chapter 2]{jantzen}, or \cite[Sections 4.1 and 15.3]{FultonHarris}.

 We denote by $X^\ast:=X^\ast(\Torus)$\index{$X^\ast$} the group of characters of $\Torus$.  Via $\Borel/\Nilp\cong \Torus$, we also view $X^\ast(\Torus)$ as characters on $\Borel$.
 Given $\kappa\in X^\ast(\Torus)$ and a $T$-module $M$, we denote by $M[\kappa]$ the $\kappa$-eigenspace of $M$.  We define\index{$X^+(\Torus)$} 
 \begin{align*}
 X^+(\Torus):=\left\{\left(\kappa_{1, \tau}, \ldots, \kappa_{a_\tau, \tau}\right)_{\tau\in\T_\cmfield}\in\prod_{\tau\in\T_\cmfield}\ZZ^{a_\tau, \tau}|\kappa_{i, \tau}\geq \kappa_{i+1, \tau} \mbox{ for all } i\right\}.
 \end{align*}
We identify $X^+(\Torus)$ with the subgroup of $X^\ast(\Torus)$ of {\it dominant} weights in $X^\ast(\Torus)$ via
\begin{align*}
\prod_{\tau\in\T_\cmfield}\diag\left(t_{1, \tau}, \ldots, t_{a_\tau, \tau}\right)\mapsto\prod_{\tau\in\T_\cmfield}\prod_{1\leq i\leq a_\tau}t_{i, \tau}^{\kappa_i, \tau}.
\end{align*}
If $\kappa = \left(\kappa_\tau\right)_{\tau\in\Sigma_\cmfield}=\left(\kappa_{1, \tau}, \ldots, \kappa_{n, \tau}\right)_{\tau\in\Sigma_\cmfield}$ is a dominant weight of $\GL_m$ and $n>m$, then we denote by $(\kappa,0)$ the dominant weight $(\kappa_{1, \tau},\dots, \kappa_{n,\tau}, 0,\dots, 0)_{\tau\in\Sigma_\cmfield}$ of $\GL_n$.
Given $k\in\ZZ$, we denote by $\underline{k}$\index{$\underline{k}$} the element $\kappa\in X^+(\Torus)$ such that
 $\kappa_{i, \tau} = k$ for all $i, \tau$.  We call $\underline{k}$ a {\it parallel, scalar} weight, and in this case, we also sometimes just write $k$ for the weight.  More generally, if for each $\tau$, there exist $k_\tau\in\ZZ$ such that $\kappa_{i, \tau} = k_\tau$ for each $i$, then we call $\kappa = \left(\left(\kappa_{\tau, i}\right)_i\right)_{\tau}$ a {\it scalar} weight, and we write $\underline{k}_\tau :=\kappa_\tau = \left(k_\tau, \ldots, k_\tau\right)$.  In this case, we also sometimes just write $k_\tau$ for the weight at $\tau$.  Also, if $\kappa = \left(\kappa_\tau\right)_\tau$ is such that $\kappa_\tau = \kappa_\sigma$ for all $\tau, \sigma$, we say that $\kappa$ is {\it parallel}.
If $\kappa\in X^+(\Torus)$ is such that $\kappa\neq 0$ and
 $\kappa_{i, \tau}\geq 0$ for all $i, \tau$, we say that $\kappa$ is {\it positive}.  If $\kappa$ is positive and $\tau\in\T_\cmfield$, then we say $\kappa$ {\it is supported at $\tau$}
 if $\kappa_{i, \tau}\neq 0$ for some $i$ and $\kappa_{j, \sigma}=0$ for all $\sigma\neq \tau$ and all $j$. 
 For $\tau\in\T_\cmfield$ and positive $\kappa$, we define
 \begin{align}
d_{\kappa, \tau}&:=\left|\kappa_\tau\right|:=\sum_{i=1}^{a_\tau}\kappa_{i, \tau}\label{sizektdefn}\\
d_\kappa&:= |\kappa|:=\sum_{\tau\in\T_\cmfield}d_{\kappa, \tau}=\sum_{\tau\in\T_F}\left|\kappa_\tau\right|.\nonumber
 \end{align}
We say $\kappa$ is {\it sum-symmetric at $\tau$} if $\kappa$ is positive and $d_{\kappa, \tau} = d_{\kappa, \tau^\ast}$.  We say $\kappa$ is {\it sum-symmetric} if $\kappa$ is positive and sum-symmetric at all $\tau\in\Sigma_{\cmfield}$.  Given $\tau\in\Sigma_\cmfield$ and $\kappa\in X^+(\Torus)$, if $\kappa_{i, \tau}=\kappa_{i, \tau^\ast}$ for all $i\leq\min\left(a_\tau, a_{\tau^\ast}\right)$ and $\kappa_{i, \tau}, \kappa_{i, \tau^\ast} = 0$ for all $i>\min\left(a_\tau, a_{\tau^\ast}\right)$, we say that $\kappa$ is {\it symmetric at $\tau$}.  If $\kappa$ is symmetric at each $\tau\in\Sigma_{\cmfield}$, we say that $\kappa$ is {\it symmetric}.  This is the same condition on weights that occurs without a name in \cite[Theorem 12.7]{shar} and  \cite[Theorem 2.A]{shclassical}.

To each dominant weight $\kappa$, we associate a representation $\rho_\kappa$\index{$\rho_\kappa$} obtained by application a $\kappa$-Schur functor $\schur_\kappa$.\index{$\schur_\kappa$} (See, e.g., \cite[Section 15.3]{FultonHarris}, for details on Schur functors.)  Let $R$ be a $\ZZ_p$-algebra or a field of characteristic $0$, and $V:=V_R:= \oplus_{\tau\in\T_\cmfield} (R^{a_\tau})$ denote the standard representation of $\prod_{\tau\in\T_\cmfield}\GL_{a_\tau}$ over $R$.
If $\kappa$ is a dominant weight, the $\kappa$-Schur functor acts on $R$-modules so that we obtain a representation $\schur_\kappa (V_R)$ 
  of $\prod_{\tau\in\T_\cmfield}\GL_{a_\tau}$, which we denote by $\rho_\kappa:=\rho_{\kappa, R}$.  
 As explained in \cite[Chapter II.2]{jantzen}, if $R$ is furthermore of sufficiently large characteristic or of characteristic $0$, then each representation $\rho_{\kappa, R}$ is irreducible, and furthermore, the set of representations $\rho_{\kappa, R}$ is in bijection with the set of dominant weights $\kappa$.  Following the conventions of \cite[Section 2.3]{EiMa}, when $R$ is such a field and has ring of integers $\CO$, we denote by $\rho_{\kappa, \CO}$ a choice of a $\CO$-lattice in $\rho_{\kappa, R}$.
Also, given a locally free sheaf of modules $\mathcal{F}$ over a $\ZZ_p$-scheme $T$, we write $\schur_\kappa(\mathcal{F})$ for the locally free sheaf of modules over $T$, defined by $\schur_\kappa (\mathcal{F})({\rm Spec} R)=\schur_\kappa (\mathcal{F}({\rm Spec} R))$, for ${\rm Spec} R$  any affine open of $T$.

For each positive dominant weight $\kappa$, by applying a generalized Young symmetrizer, we obtain a projection \index{${\rm pr}_\kappa$}${\rm pr}_\kappa: V^{\otimes d_\kappa}\rightarrow \rho_\kappa$.
 If $\kappa_\tau$ is a positive, dominant weight and $R$ is as above, then the $\kappa_\tau$-Schur functor is $\schur_{\kappa_\tau}(V):=V^{\otimes d_{\kappa, \tau}}\cdot c_{\kappa, \tau}$, where $c_{\kappa, \tau}$ denotes the Young symmetrizer associated to $\kappa_\tau$.  
 As noted in \cite[Lemma 2.4.6]{EFMV}, if $\kappa,\kappa'$ are two positive, dominant weights, then ${\rm pr}_{\kappa+\kappa'}$ factors through the map ${\rm pr}_\kappa\otimes {\rm pr}_{\kappa'}$; we write
\index{${\rm pr}_{\kappa,\kappa'}$}${\rm pr}_{\kappa,\kappa'}$ for the induced projection $\rho_\kappa\otimes \rho_{\kappa'}\to \rho_{\kappa+\kappa'}$.

\begin{defi}\label{twistweight}
For any dominant weight $\kappa=(\kappa_\tau)_{\tau\in\T_F}$, we write 
\begin{align*}
||\kappa||:=(||\kappa_\tau||)_{\tau\in\T}\in \ZZ^{|\T_F|}, 
\end{align*}
where $||\kappa_\tau||\in\ZZ$ is defined as in
\begin{align*}
||\kappa_\tau||:=
\begin{cases}
 |\kappa_\tau|/a_\tau & \mbox{ if $\kappa_\tau$ is scalar}\\
|\kappa_\tau|& \mbox{ otherwise.}
\end{cases}
\end{align*}
(By Equation \eqref{sizektdefn}, if $\kappa_\tau$ is scalar, then $\left|\kappa_\tau\right|$ is a multiple of $a_\tau$.)

Note that $||\kappa_\tau||$  is the unique integer satisfying, 
for $\chi$ any character of $GL_{a_\tau}$,
\begin{align*}
\schur_{\kappa_\tau}(V_\tau)\otimes\chi^{||\kappa_\tau|| }\simeq
\begin{cases}
\schur_{||\kappa_\tau||}\left(\det (V_\tau)\otimes\chi\right) & \mbox{ if $\kappa_\tau$ is scalar}\\
\schur_{\kappa_\tau}(V_\tau\otimes\chi) & \mbox{ otherwise.}
\end{cases}
\end{align*}
\end{defi}

\begin{defi}\label{determinantpower}
For any dominant weight $\kappa=(\kappa_\tau)_{\tau\in\T_F}$, we write \[r(\kappa):=(r(\kappa_\tau))_{\tau\in\T_{\cmfield} }\in\ZZ^{|\T_\cmfield |},\] where \index{$r(\kappa)$}$r(\kappa_\tau)\in\ZZ$ is defined as in \[r(\kappa_\tau):=|\kappa_\tau|\cdot\dim\rho_{\kappa_\tau}/a_\tau.\]
(By definition, if $\kappa$ is scalar, then $r(\kappa)=||\kappa||$.)

Note that $r(\kappa_\tau)$ is the unique integer satisfying the equality
\[\det(\schur_{\kappa_\tau}(f))=\det(f)^{r(\kappa_\tau)}\]
for  $f$ any linear endomorphism of the standard representation of $GL_{a_\tau}$.

\end{defi}

Following the convention of \cite[Definition 2.2.3]{EFGMM}, given a positive integer $e$, we call a dominant weight $\kappa$  {\it admissible of depth $e_\kappa=e$} if the irreducible representation $\rho_\kappa$ of $\Levi$ occurs as a constituent of the representation $\left(V^2\right)^{\otimes e}$ for
\begin{align}\label{exp2-notation}
V^2:=\begin{cases}\oplus_{\tau\in\T_{\realfield}}\Sym^2 V_\tau,& \mbox{ in the symplectic case (C)}\\
\oplus_{\tau\in\Sigma_\cmfield} V_\tau\boxtimes V_{\tau^\ast},& \mbox{ in the unitary case (A)}
\end{cases}
\end{align}
In the above cases, we also define $V_\tau^2$ to be the summand at $\tau$.  By abuse of language, we also speak of being admissible of depth $e$ at $\tau$.  We denote by $\delta(\tau)$\index{$\delta(\tau)$} the weight of $V_\tau^2$. Admissible weights are even in the symplectic case,
and sum-symmetric in the unitary case.

\begin{rmk}\label{sym_rem}
Here, we will be particularly interested in the case of irreducible constituents that arise inside symmetric powers of $V^2$ and, more generally, inside $\boxtimes_\tau \Sym^{e_\tau}\left(V_\tau^2\right)$ for $e_\tau\geq 0$ integers.  
By \cite[Theorem 2.A]{shclassical}, such constituents are symmetric and occur with multiplicity one.  In the case where $a_\tau=a_{\tau^\ast}$ for all $\tau\in\Sigma_\cmfield$, this is the Peter--Weyl Theorem (see, e.g., \cite[Theorem 4.66]{etingof}).
\end{rmk}

\begin{rmk}
Let $\kappa=(\kappa_\tau)_{\tau\in\T}$ be a dominant weight.  
Note that $\rho_\kappa$ is one-dimensional if and only if $\kappa$ is a scalar weight, i.e. $\kappa = \left(\underline{k}_\tau\right)_\tau$ for some nonnegative integers $k_\tau\in\ZZ$.  In this case $\rho_\kappa = \boxtimes_\tau{\det}^{k_\tau}$, i.e. the $k_\tau$-th powers of the top exterior powers.
\end{rmk}

\begin{rmk}\label{divideweight}
Let $R$ be a $\ZZ_p$-algebra or a field of characteristic $0$, and $\Levii$ a Levi subgroup of $\Levi$ as in Section \ref{weights-subgroups}.
For any dominant weight $\kappa'$ of $\Levii$, we denote by $\varrho_{\kappa'}=\varrho_{\kappa',R}$ the irreducible algebraic representation of $\Levii$ over $R$, of highest weight $\kappa'$. If $R$ is of  sufficiently large characteristic  or of characteristic $0$, then for any dominant weight $\kappa$ of $\Levi$, we identify 
\[\rho_\kappa\vert_{\Levii}=\bigoplus_{\kappa'\in\mathfrak{M}_\kappa} \varrho_{\kappa'},\]
where  $\mathfrak{M}_\kappa$ denotes the set of all dominant weights $\kappa'$ of $\Levii$ occurring in $\rho_\kappa\vert_{\Levii}$ (see \cite[Section 2.4]{EiMa}).
\end{rmk}

 \subsubsection{Automorphic forms}\label{autoforms-section}
We recall the construction of automorphic forms on $\Sh$, following the approach of \cite[Section 3.2]{CEFMV}. {\bf Since we are working in the setting of automorphic forms, all the weights that will arise for us are positive and dominant.  Thus, going forward, we only consider positive, dominant weights.}

We denote by $\alpha: \Auniv\rightarrow \Sh$\index{$\Auniv$} the universal abelian scheme and by $\uo$\index{$\uo$} the sheaf $\uo_{\Auniv/\Sh}:=\alpha_\ast\Omega_{\Auniv/\Sh}^1$.\footnote{There are several conventions for the sheaf $\uo$ and closely related sheaves in the literature.  In some of the first named author's prior works, this sheaf was denoted by $\underline{\omega}$.  To avoid confusion with \cite{EiMa}, where $\underline{\omega}$ had a different meaning, we avoid that notation here.  We also note that, in contrast to the present paper, \cite{GN} and some other references denote by $\omega$ the top exterior power of the sheaf $\alpha_\ast\Omega_{\Auniv/\Sh}^1$, but we will explicitly denote the top exterior power as such when we need to take it.}  The sheaf $\uo$ is locally free of rank
\begin{align*}
g:=n[\realfield:\IQ]
\end{align*}
and decomposes, according to Section \ref{signature-intro}, as
\begin{align*}
\uo = \oplus_{\tau\in\Sigma_\cmfield}\left(\uo_\tau\oplus\uo_{\tau^\ast}\right),
\end{align*}
with $\uo_\tau$ (resp. $\uo_{\tau^\ast}$) locally free of rank $a_\tau$ (resp. $a_{\tau^\ast}$).
Consider the locally free sheaves\index{$\CE$} 
\begin{align*}
\CE_\tau&:=\Isom\left(\uo_\tau, \CO_{\Sh, \tau}^{a_\tau}\right)\\
\CE&:=\oplus_{\tau\in\T_\cmfield}\CE_\tau
\end{align*}
 endowed with an action of $\Levi = \prod_{\tau\in\T_\cmfield}\Levi_\tau.$  
 For each positive dominant weight $\kappa$ of $\Torus$ and each irreducible representation $\left(\rho_\kappa, V_\kappa\right)$ of $\Levi$ of weight $\kappa$, the {\it sheaf of weight $\kappa$ (or weight $\rho_\kappa$) automorphic forms} is\index{$\uo^\kappa$}
\begin{align*}
\uo^\kappa:=\CE\times^\Levi V_\kappa
\end{align*}
defined so that
\begin{align}\label{uokap-defn}
\uo^\kappa(R):=\left(\CE\times V_\kappa\otimes R\right)/\left(\left(\ell, m\right) \sim  \left(g\ell, \rho_\kappa\left({ }^tg^{-1}\right)m\right)\right)
\end{align}
for each $\CO_{\reflexfield, \mathfrak{p}}$-algebra $R$.  (N.B. This is closely related to the notion of a {\it frame bundle}.)
Given an $\CO_{\reflexfield, \mathfrak{p}}$-algebra $R$, an {\it automorphic form of weight $\kappa$ and level $K$}, defined over $R$, is a global section of $\omega^\kappa$ on $\Sh_K\times_{\CO_{\reflexfield, \mathfrak{p}}} R.$  
As noted in \cite[Section 2.5]{EiMa},  
$\uo^\kappa$ can be canonically identified with $\schur_\kappa(\uo)$.
Note, also, that if $\kappa = \left(\kappa_\tau\right)_\tau$, then
\begin{align*}
\uo^\kappa = \boxtimes_\tau\uo_\tau^{\kappa_\tau}.
\end{align*}

\begin{rmk}
Excluding the one-dimensional case of $\realfield=\IQ$ with $a_\tau=a_\tau^\ast=1$  (no loss to the present paper, which aims to overcome technical challenges with extending to higher rank the sorts of results that have already been established in low rank), the Koecher principle (\cite[Theorem 2.3 and Remark 10.2]{lan5}) implies that our space of automorphic forms is the same as the one obtained by instead working over a compactification of $\Sh_K$.
\end{rmk}

\subsection{The $\mu$-ordinary locus and its Hasse invariant}\label{hasselocus_sec}
\label{muordinary-bkgd}
We now recall the definitions and key features  of the $\mu$-ordinary locus (following \cite{wedhorn, moonen}) and of the $\mu$-ordinary Hasse invariant (following \cite{GN}) for PEL-type Shimura varieties.
For generalizations of these notions and key results to the context of Hodge-type Shimura varieties, the reader may refer to \cite{KWhasse, wortmann}, although we shall not need them in the present paper.

\subsubsection{The $\mu$-ordinary Newton polygon stratum} We briefly recall the definition and key features of the $\mu$-ordinary Newton polygon stratum $\shpmuord$\index{$\shpmuord$} of $ \Shp$ (see also \cite[Section 2.6]{EiMa}).

\begin{defi}\label{slopes}\label{muNP}
Given the Shimura datum $\shimuradatum$, the {\em $\mu$-ordinary Newton polygon} at $p$, denoted by $\nu_p(\shimuradatum):=\nu_p(n,\cf)$,\index{$\nu_p(\shimuradatum)$}\index{$\nu_p(n,\cf)$} is defined as the amalgamate sum $\nu_p(n,\cf)=\oplus_{\co\in\mathfrak{O}_F}\nu_{\co}(n,\cf)$, where 
for each $\co\in\mathfrak{O}_F$, $\nu_\co(n,\cf)$\index{$\nu_{\co}(n,\cf)$} is the polygon with slopes\index{$a^\co_j$} \[a^\co_j:=\frac{\#\{\tau\in \co |\cf(\tau)>n-j\}}{\#\co}, \text{ for } j=1,\dots, n.\]
\end{defi}

By construction, the $\mu$-ordinary Newton polygon $\nu_p(n,\cf)$ is the lowest Newton polygon at $p$ compatible with the signature $(n,\cf)$ of the Shimura datum.  We say that the polygon $\nu_p(n,\cf)$ is {\em ordinary} if  $a_j^\co \in\{ 0,1\}$ for $j=1, \ldots, n$ (in which case it corresponds to an ordinary abelian variety).  So the polygon $\nu_p(n,\cf)$ is ordinary if and only if $(p)$ is totally split in $\reflexfield$. In the following, abusing notation, we put \index{$\nu(n,\cf)$}$\nu(n,\cf):=\nu_p(n,\cf)$.

\begin{defi}
A $\mathfrak D$-enriched abelian variety $A$  (resp. a point of $\Shp$) over a field containing $\F$ is called {\em $\mu$-ordinary} if its Newton polygon (resp. the Newton polygon of the associated abelian variety) agrees with the $\mu$-ordinary Newton polygon $\nu(n,\cf)$.\end{defi}

By definition,  the $\mu$-ordinary Newton polygon stratum $\shpmuord$ is the reduced subscheme of $\Shp$ consisting of all $\mu$-ordinary points. 

\begin{thm}\cite[(1.6.2) Density Theorem]{wedhorn}
The $\mu$-ordinary Newton polygon stratum $\shpmuord$ is open and dense  in $\Shp$. 
\end{thm}

\begin{thm}\cite[Theorem 3.2.7]{moonen} The $\mu$-ordinary Newton polygon stratum is also an Ekedahl--Oort stratum. I.e., there exists a unique up to isomorphism $\mu$-ordinary $\mathfrak D$-enriched truncated \BT group of level 1over $\mathbb F$.
 Furthermore,  there exists a unique up to isomorphism $\mu$-ordinary $\mathfrak D$-enriched \BT group over $\mathbb F$.\end{thm}

The latter result is the key ingredient in the construction of the $\mu$-ordinary Hasse invariant  in \cite{GN}.

\subsubsection{$\mu$-ordinary Hasse invariant }\label{hasse_sec}
 We now recall the definition  and the key features of the $\mu$-ordinary Hasse invariant (see also \cite[Section 2.7]{EiMa}).
In the following, $|\omega|$ denotes the Hodge line bundle over $\Sh$:\index{$|\omega|$}
\begin{align*}
|\omega|:=\wedge^{\rm top} \omega_{\CA/\Sh}, 
\end{align*}
where $\wedge^{\rm top}$ denotes the top exterior power.

\begin{thm}\cite[Theorem 1.1]{GN}\label{hasse_thm}
There exists an explicit positive integer $m_0\geq 1$,\index{$m_0$} and a section $$\HA_\mu\in H^0(\Shp,|\omega|^{m_0})$$\index{$\HA_\mu$}
such that: 
\begin{enumerate}
\item The non-vanishing locus of $\HA_\mu$ is the $\mu$-ordinary locus of $\shpmuord$.
\item The construction of $\HA_\mu$ is compatible with varying the level $\compact^{(p)}$.
\item The section $\HA_\mu$ extends to the minimal compactification of $\Shp$.
\item A power of $\HA_\mu$ lifts to characteristic zero.
\end{enumerate}
\end{thm}

By construction  (\cite[Definition 3.5]{GN}),
$m_0:={\rm lcm}_{\tau\in\T_F}(p^{e_\tau}-1)$, where \index{$e_\tau$}$e_\tau:=\#\co_\tau$, for  $\tau\in\T_F$.

In \cite[Definition 3.5]{GN}, Goldring--Nicole define the $\mu$-ordinary Hasse invariant $\HA_\mu$ (in {\em loc. cit.} denoted by $^\mu H$)  as\index{$m_\tau$} \[\HA_\mu:=\prod_{\tau\in\T_F} \HA_\tau^{m_\tau},
\text{ where } m_\tau=\frac{\mathrm{lcm}_{\tau'\in\T_F}(p^{e_{\tau'}}-1)}{p^{e_\tau}-1},\]
where $\HA_\tau \in H^0(\Shp,|\omega_\tau|^{p^{e_\tau}-1})$ denotes the $\tau$-Hasse invariant  (\cite[Definition 3.3]{GN}, in {\em loc.cit.} denoted by $^\tau H$) and the product is over all elements of $\T_F$.  In the following, for each $\tau\in\T_F$, we denote the weight of the $\tau$-Hasse invariant by \index{$\kappa_{\ha,\tau}$}$\kappa_{\ha,\tau}$; it is the scalar weight $(p^{e_\tau}-1)$ supported at $\tau$.  

For any subset $\Sigma\subseteq \T_F$, we define the $\Sigma$-Hasse invariant as \index{$\HA_\Sigma$}\[\HA_\Sigma:= \prod_{\tau\in\Sigma} \HA_\tau\in H^0(\Shp,\omega^{\kappa_{\ha,\Sigma}}),\] where by definition 
\begin{align*}
(\kappa_{\ha,\Sigma})_\tau:=
\begin{cases}
\kappa_{\ha,\tau} & \mbox{ for }\tau\in\Sigma\\
0 &  \mbox{ otherwise.} 
\end{cases}
\end{align*}
In particular, the weight $\kappa_{\ha,\Sigma}$\index{$\kappa_{\ha,\Sigma}$} is scalar and supported at $\Sigma$.
In the following, for $\Sigma=\T_F$, we write $\HA:=\HA_{\T_F}$ of scalar weight  \index{$\kappa_{\ha}$}$\kappa_{\ha}:=\kappa_{\ha,\T_F}=(\kappa_{\ha,\tau})_{\tau\in\T_F}$.

Finally, for any $\underline{b}=(b_\tau)_{\tau\in\T}\in\ZZ^{|\T|}$, we write  $\underline{b}\cdot \kappa_\ha$, or just $\underline{b} \kappa_\ha$, for the scalar weight $\left(b_\tau (p^{e_\tau}-1)\right)_{\tau\in\T}$, and we set 
\[\HA^{\underline{b}}:=\prod_{\tau\in\T_F} \HA_\tau^{b_\tau} \in H^0(\Shp,\omega^{\underline{b}\cdot \kappa_\ha}).\]
With these conventions, the $\mu$-ordinary Hasse invariant \index{$\HA_\mu$}$\HA_\mu=\HA^{\underline{m_0}}$ is a scalar-valued automorphic form of {\em parallel} weight
$\underline{\mathrm{ lcm}_{\tau\in\T_F}(p^{e_\tau}-1)}$.

In the following, we denote by $\shmuord$\index{$\shmuord$} the formal completion of $\Sh$ along $\shpmuord$.  We follow the convention of (abusing language and) referring to the formal scheme $\shmuord$ as the $\mu$-ordinary locus over $\Witt$.

\subsection{Automorphic sheaves over the $\mu$-ordinary locus}\label{muord_sec}
We briefly recall previous results  of the restriction  of automorphic sheaves to $\shmuord$.
We refer to \cite[Sections 3,4, and 6]{EiMa} for details.

\subsubsection{Slope filtration and associated graded module}
Let $\CA[p^\infty]$ denote the $p$-divisible part of the universal abelian scheme $\CA$ over $\Sh$. 
For any $\sigma$-orbit $\co$ in $\T_F$, we denote by \index{${\mathfrak p}_\co$}${\mathfrak p}_\co$ the associated prime of $F$ above $p$, and write \index{$\CG_\co$}$\CG_\co=\CA[{\mathfrak p}^\infty_\co]$. Hence, $\CA[p^\infty]=\oplus_{\co\in\mathfrak{O}_F}\CG_\co$.

By \cite[\S 3]{Man05} (see also \cite[Proposition 3.1.1]{EiMa}), the restriction to $\shmuord$ of
$\CA[p^\infty]$, resp. $\CG_\co$ for any $\co\in\mathfrak{O}_F$, is completely slope divisible, and it admits a slope filtration over $\shmuord$, which we denote by $\CA[p^\infty]_\bullet$, resp. $\CG_{\co \bullet}$. (Note that  the slope filtration of $\CG_\co$ agrees with the filtration induced by the slope filtration of $\CA[p^\infty]$.)
We write $\gr(\CA[p^\infty])$ (resp. $\gr(\CG_\co)$) for the associated graded $\mathfrak{D}$-enriched \BT group.

The slope filtration of $\CA[p^\infty]$ induces a filtration $\uo_\bullet$ of the  $\CO_\shmuord \otimes_\Witt\CO_F$-module $\uo$, and a filtration  $\uo_{\co\bullet}$ of the $\CO_\shmuord \otimes_\Witt\CO_{F,\mathfrak{p}_\co}$-module $\uo_\co$. 
We denote by $\underline{\uo}:=\gr(\uo)$, resp. $\underline{\uo}_\co:=\gr(\uo_\co)$, the associated graded sheaf over $\shmuord$. The sheaf $\underline{\uo}$ is a locally free $\CO_\shmuord \otimes_\Witt\CO_F$-module, and  $\underline{\uo_\co}$ is a locally free $\CO_\shmuord \otimes_\Witt\CO_{F,\mathfrak{p}_\co}$-module. 

For each $\tau\in\T_\cmfield$, let $s_\tau$ denote the number of distinct slopes of the polygon $\nu_{\co_\tau}(n,\cf)$ (see Definition \ref{muNP}),  and set for $1\leq t\leq s_\tau$
\begin{align}\label{ordered_partition}
m_{t,\tau}:=\rk (\gr^t(\uo_\tau)).
\end{align}
By construction,  $m_{t,\tau}\geq 0$  for all $1\leq t\leq s_\tau$, and $m_{1,\tau}+\cdots +m_{s_\tau,\tau}=a_\tau$  for all $\tau\in\T_\cmfield$ (see \cite[\S 4.2]{EiMa} for an explicit description of the partition given the Shimura datum).

\subsubsection{OMOL sheaves}\label{omol-bkgd}
Recall the subgroups $\Para, \Uni, \Levii$ of $\Levi$ introduced in Section \ref{weights-subgroups}.  In \cite[Definition 4.2.1]{EiMa}, for any positive dominant weight $\kappa'$ of $\Levii$,  we introduce the sheaf $\underline{\uo}^{\kappa'}:=\gr(\uo)^{\kappa'}$ over $\shmuord$, which we call an
{\it OMOL} sheaf.
In \cite[Proposition 4.3.1]{EiMa}, for any weight $\kappa$ of $\Levi$, we describe the restriction  to $\shmuord$ of the automorphic sheaf $\uo^\kappa$ in terms of these auxiliary sheaves.  More precisely,  we prove the following result.

\begin{prop}\cite[Proposition 4.3.1]{EiMa}\label{filtration_prop}
Let $\kappa$ be a  weight of $\Levi$.  The sheaves, $\uo$ and $\uo^\kappa$ are defined over $\shmuord$.  Moreover, each of the following holds.
\begin{enumerate}
\item Each standard $\Uni$-stable filtration of $\rho_{\kappa\vert \Para}$  induces a filtration on $\uo^\kappa$.
\item The sheaf  $\gr(\uo^\kappa)$ is independent of the choice of a standard filtration on $\rho_{\kappa\vert \Para}$.
More precisely, there is a canonical identification\index{$\iota_\kappa$}
 \[\iota_\kappa:\gr(\uo^\kappa)\simeq  \oplus_{\kappa'\in\mathfrak{M}_\kappa}\underline{\uo}^{\kappa'},\] for $\mathfrak{M}_\kappa$ as in Remark \ref{divideweight}.
 \item There is a canonical projection \index{$\varpi^\kappa$}$\varpi^\kappa:\omega^\kappa\to\underline{\uo}^\kappa$, which is an isomorphism if $\kappa$ is scalar. 
\end{enumerate}\end{prop}

\subsection{Hecke algebras and Galois representations}\label{galois_sec}
We now recall the Galois representations (conjecturally) associated to mod $p$ Hecke eigenforms, following the setup of \cite[Section 2.1]{EFGMM}. 
We refer to \cite{Gross-satake, buzzgee14} for details.

Throughout this section, let $\randomfield$ be a field over which $G$ is split, i.e. every maximal torus of $G$ is isomorphic over $\randomfield$ to a product of copies of $\Gm$.  Suppose also that $\Torus$ and $\Borel$ are defined over $\randomfield$.  Although we have specified above that $G$ is of unitary or symplectic type, here we merely assume that $G$ is a connected and reductive group over $\IQ$.  As introduced in Section \ref{groups-intro}, we continue to denote by $\compact$ a level of $G$ that is neat and such that $\compact_p\subset G\left(\IQ_p\right)$ is hyperspecial.  Note that even 
with the relaxed assumptions on  $G$, $\Sigma_{\compact, p}$ is finite.  

\subsubsection{Local Hecke algebras}
Suppose that $v$ is a finite place of $\randomfield$ and that the completion $\randomfield_v$ of $\randomfield$ at $v$ is a nonarchimedean local field.  Denote by $\CO_v$ the ring of integers of $\randomfield_v$, $\varpi_v\in\CO_v$ a choice of uniformizer, and $q_v$ the cardinality of $\CO_v/\varpi_v\CO_v$.  Suppose that $G$ is split over $\CO_v$.  Then there exists a group scheme $\G$ over $\CO_v$
with generic fiber $G$ and reductive special fiber.  We denote by $G_v$ the group of points $\G\left(\randomfield_v\right)$.  We additionally choose $v$ so that $\compact_v$ is a hyperspecial maximal compact subgroup of $G_v$.

Given a commutative ring $R$, the {\it local Hecke algebra} of $\left(G_v, \compact_v\right)$ is the $R$-algebra
\begin{align*}
\Heckealgebra\left(G_v, \compact_v; R\right):=\left\{h: \compact_v\backslash G_v/\compact_v \rightarrow R \mid h \mbox{ is locally constant and compactly supported}\right\}
\end{align*}
with multiplication defined by
\begin{align*}
\left(h_1* h_2\right)\left(\compact_v g\compact_v\right):= \sum_{x\compact_v\in G_v/\compact_v}h_1\left(\compact_v x\compact_v\right)h_2\left(\compact_v x^{-1}g\compact_v\right).
\end{align*}
The ($\bmod p$) Satake transform is a ring isomorphism
\begin{align*}
\Satake_v:\Heckealgebra\left(G_v, \compact_v; \Fbar_p\right)\rightarrow R\left(\hat{G}\right)\otimes\Fbar_p,
\end{align*}
where 
$R\left(\hat{G}\right)$ is the representation ring of $\hat{G}$, the dual group of $G$.  

The characters $\omega$ of $R\left(\hat{G}\right)\otimes\Fbar_p$ are indexed by the semi-simple conjugacy classes  $s\in\hat{G}\left(\Fbar_p\right)$, via
\begin{align*}
s\leftrightarrow \left(\omega_s\left(\chi_\rho\right) := \chi_\rho(s)\right),
\end{align*}
where  $\chi_\rho:=\Tr\left(\rho\right)$, for $\rho$ any irreducible representation of $\hat{G}$.

\subsubsection{Galois representations associated to $\bmod p$ Hecke eigenforms}
Let $f$ be a mod $p$ Hecke eigenform on $G$ of level $\compact$ defined over $\Fbar_p$ (i.e. as introduced in Section \ref{autoforms-section}, a global section of the vector bundle over the associated Shimura variety).
Associated to the Hecke eigenform $f$, and a finite place $v$ of $L$ as above, we have a Hecke eigensystem
\begin{align*}
\Psi_{f, v}: \Heckealgebra\left(G_v, \compact_v; \Fbar_p\right)\rightarrow\Fbar_p
\end{align*}
defined by
\begin{align*}
Tf = \Psi_{f, v}(T) f
\end{align*}
for each $T\in\Heckealgebra\left(G_v, \compact_v; \Fbar_p\right)$. 
The {\it $v$-Satake parameter} of $f$ is the semi-simple conjugacy class $s_{f, v}\in \hat{G}\left(\Fbar_p\right)$ indexing the character $\omega_{f,v}$,\begin{align*}
\omega_{f,v}:=\Psi_{f, v}\circ \Satake_v^{-1}:R\left(\hat{G}\right)\otimes\Fbar_p\rightarrow\Fbar_p,
\end{align*}
i.e. $\omega_{s_{f, v}} = \omega_{f,v}$.

\begin{conj}[Positive characteristic form of Conjecture 5.17 of \cite{buzzgee14}]\label{gal-rep-conj}
There exists a continuous representation 
\begin{align*}
\rho: \Gal\left(\overline{\randomfield}/\randomfield\right)\rightarrow \hat{G}\left(\Fbar_p\right),
\end{align*}
unramified outside $\Sigma_{\compact, p}$, such that for all $v\nin\Sigma_{\compact, p}$, the image of the Frobenius element $\Frob_v$ at $v$ is $\rho\left(\Frob_v\right)=s_{f, v}.$
\end{conj} 
By \cite[Remark 5.19]{buzzgee14}, the set of Galois representations $\rho$ associated to $f$ as in Conjecture \ref{gal-rep-conj}  
is not necessarily finite.  
A comparison of the formulation in Conjecture \ref{gal-rep-conj} with the original statement from \cite{buzzgee14} is provided in \cite[Remark 2.1.2]{EFGMM}.  As an aside, we note that it might be possible to further describe the representations $\rho$ from Conjecture \ref{gal-rep-conj} (e.g. as odd); but since the present paper studies the effect of theta operators, such details would not impact our results.

\section{Some differential operators}\label{DiffOpsReview-section}

In this section, we recall the construction of weight-raising differential operators on $p$-adic automorphic forms, which arise as analogues of Maass--Shimura differential operators.  
By construction, these operators raise the weight of the automorphic forms by admissible weights. In Proposition \ref{prop-sym}, 
 we observe that, as it is the case for classical Maass--Shimura operators, {\bf the $p$-adic differential operators are non-trivial only for symmetric weights. 
Hence, going forward, we will only consider differential operators associated with symmetric weights.}
In Section \ref{frobsplit_sec}, we discuss some preliminary observations on new phenomena which arise in positive characteristic at primes when the ordinary locus is empty.

\subsection{Gauss--Manin connection and Kodaira--Spencer morphism}
Given a $\CO_{\realfield^\Gal}$-algebra $R$ in which $d_{\realfield}$ is invertible, a scheme $T=\Spec R$, a smooth morphism of schemes $Y\rightarrow T$, and a polarized abelian scheme $\cp: A\rightarrow Y$ together with an action of $\CO_{\cmfield}$ (e.g. when $A$ is an abelian variety parametrized by a unitary or symplectic Shimura variety), consider the Hodge filtration
\begin{align*}
0\rightarrow \uo_{A/Y}:=\cp_\ast\Omega^1_{A/Y}\hookrightarrow\hdrone\left(A/Y\right)\rightarrow R^1\cp_\ast\CO_A\rightarrow 0.
\end{align*}
As in \cite{kaCM, hasv, EDiffOps, EFGMM, EiMa}, we build differential operators from the Gauss--Manin connection\index{$\nabla$}
\begin{align*}
\nabla:=\nabla_{A/Y}:\hdrone\left(A/Y\right)\rightarrow\hdrone\left(A/Y\right)\otimes\Omega^1_{Y/T}
\end{align*}
and the Kodaira--Spencer morphism\index{$\KS$}
\begin{align*}
\KS:=\KS_{A/Y}:\uo_{A/Y}\otimes\uo_{A/Y}\rightarrow\Omega_{Y/T}^1.
\end{align*}
(Details on the Gauss--Manin connection and the Kodaira--Spencer morphism are available in, e.g., \cite[Sections 2.1.7 and 2.3.5]{la} and \cite[Section 9]{FaltingsChai}.)  

Via the product rule (i.e. Leibniz rule), for any nonnegative integer $k$, we extend the Gauss--Manin connection
to a morphism\index{$\nabla_{\otimes k}$}
\begin{align*}
\nabla_{\otimes k}: (\hdrone\left(A/Y\right))^{\otimes k}\rightarrow(\hdrone\left(A/Y\right))^{\otimes k}\otimes\Omega_{Y/T}^1
\end{align*}
by
\begin{align*}
\nabla_{\otimes k}(f_1\otimes\cdots\otimes f_k) = \sum_{i=1}^k \iota_i(f_1\otimes\cdots\otimes\nabla(f_i)\otimes\cdots\otimes f_k),
\end{align*}
where $\iota_i$ is the isomorphism
\begin{align*}
\iota_i: (\hdrone)^{\otimes i}\otimes\Omega_{Y/T}^1\otimes(\hdrone)^{\otimes (k-i)}&\isomto (\hdrone)^{\otimes k}\otimes\Omega_{Y/T}^1\\
e_1\otimes \cdots \otimes e_i\otimes u \otimes e_{i+1}\otimes\cdots\otimes e_d&\mapsto e_1\otimes\cdots\otimes e_d\otimes u.
\end{align*}
This map also naturally induces a morphism on symmetric powers, exterior powers, and their compositions, similarly marked in the subscript beneath $\nabla$.  When clarification about the specific power is not needed, we abuse notation and simply write $\nabla$ without the subscript.

By definition, 
\begin{align*}
\KS:=\langle\cdot,\nabla(\cdot)\rangle_A,
\end{align*}
 where $\langle \cdot, \cdot\rangle_A$ is the pairing induced by the polarization on $A$ and extended linearly in the second variable to a pairing between $\uo$ and $\nabla(\uo)$, exploiting the fact that $\uo$ is isotropic under this pairing. 
The Kodaira--Spencer morphism induces an isomorphism\index{$\ks$}
\begin{align}\label{ks-defn}
\ks: \uo^2\isomto\Omega_{Y/T}^1,
\end{align}
where the notation $\uo^2$ follows the convention of Equation \eqref{exp2-notation}.  

\subsection{Decompositions}
By abuse of notation, we also denote by $\nabla$ the map $(\id\otimes\ks^{-1})\circ\nabla.$
Recall that by Equation \eqref{exp2-notation}, $\uo^2$ decomposes as $\uo^2 = \oplus_{\tau\in\T_{\realfield}}\uo^2_\tau$.  Similarly to \cite[Section 2.1]{kaCM}, $\nabla$ also decomposes as a sum, over $\tau\in\T_{\realfield}$, of maps\index{$\nabla_\tau$}
\begin{align*}
\nabla_\tau: \hdrone(A/Y)\rightarrow \hdrone(A/Y)\otimes\uo_\tau^2,
\end{align*}
and similarly for its extension to tensor, symmetric, and exterior powers and their compositions, as well as Schur functors.  According to the decompositions of Section \ref{signature-intro}, $H:=\hdrone\left(A/Y\right)$ decomposes as 
\begin{align*}
H = \oplus_{\tau\in\T_{\realfield}}H_\tau,
\end{align*}
and if furthermore $\cmfield\neq\realfield$ and $R$ is a $\CO_{\cmfield^\Gal}$-algebra, each $H_\tau$ decomposes as
\begin{align*}
H_\tau = H_\tau^+\oplus H_\tau^-.
\end{align*}

In Lemma \ref{nablasigmatau} and Section \ref{defDalg_sec}, we briefly abuse notation and, for convenience when dealing with two elements of $\T$ at once, denote by $\sigma$ an element of $\T$.  Frobenius does not appear in these portions, so there should be no confusion with our use of $\sigma$ to denote Frobenius elsewhere.
\begin{lem}\label{nablasigmatau}
For any $\tau, \sigma\in\T_{\realfield}$, we have:
\begin{align*}
\nabla_\tau\left(H_\sigma\right)&\subseteq H_\sigma\otimes\uo^2_\tau\\
\nabla_\tau\left(H^{\pm}_\sigma\right)&\subseteq H^{\pm}_\sigma\otimes\uo^2_\tau
\end{align*}
\end{lem}
\begin{proof}
This follows from the definition of $\nabla$, similarly to \cite[Equations (3.3) and (3.4)]{EDiffOps}.
\end{proof}
Note that the Leibniz rule (i.e. product rule) immediately extends Lemma \ref{nablasigmatau} to tensor, symmetric, and exterior powers, and their compositions, as well as Schur functors.

\begin{rmk}
Katz and Oda prove in \cite{KatzOda} that $\nabla$ is flat, i.e. integrable, when $T=\Spec k$ with $k$ a field.  In other words, 
\begin{align*}
\nabla_1\circ\nabla = 0,
\end{align*}
where
\begin{align*}
\nabla_1: \Omega_{Y/k}\otimes_{\CO_S}\hdrone(A/Y)\rightarrow\wedge^2\Omega_{Y/k}\otimes_{\CO_Y}\hdrone(A/Y)
\end{align*}
is defined by
\begin{align*}
\nabla_1(u\otimes e) = du\otimes e - u\wedge \nabla(e)
\end{align*}
for all $u\in\Omega_{Y/k}$ and $e\in \hdrone(A/Y)$, where $d$ denotes the exterior derivative on the de Rham complex.  (For convenience, we temporarily write $\Omega_{Y/k}$ on the left of $\hdrone(A/Y)$ here.)  In this case, $\hdrone(A/Y)$ has a horizontal basis for $\nabla$, i.e. a basis of sections on which $\nabla$ vanishes.
\end{rmk}

\subsection{Algebraic differential operators}\label{defDalg_sec}

The inclusion $\uo\hookrightarrow \hdrone(A/Y)$ (from the Hodge filtration) induces inclusions $\uo_\tau\hookrightarrow H_\tau$ and $\uo^{\pm}_\tau\hookrightarrow H_\tau^{\pm}$.  Thus, we get inclusions
\begin{align}
\iota_\tau: \uo_\tau^2\hookrightarrow H_\tau^2,
\end{align}
where\index{$H_\tau^2$}
\begin{align*}
H_\tau^2:=\begin{cases}\Sym^2 H_\tau & \mbox{ symplectic case}\\
H_\tau^+\boxtimes H_\tau^- & \mbox{ unitary case}
\end{cases}
\end{align*}

Similarly to \cite[Diagram (2.1.12)]{kaCM}, we now define an algebraic differential operator $D_\tau$ as the composition\index{$D_\tau$}
\begin{align*}
D_\tau:=\left(\id\otimes\iota_\tau\right)\circ \nabla_\tau.
\end{align*}
In particular, for each  weight $\kappa = \left(\kappa_\sigma\right)_\sigma$, we obtain a map
\begin{align*}
D_\tau:=D_{\kappa,\tau}: \uo^\kappa=\boxtimes_\sigma\uo_\sigma^{\kappa_\sigma}\rightarrow H^\kappa\otimes \uo_\tau^2\subseteq H^\kappa\otimes H_\tau^2\subseteq H^\kappa\otimes H^2,
\end{align*}
where $H^2=\oplus_{\tau\in\T_{\realfield}}H_\tau^2$ and $H^\kappa = \boxtimes_\sigma H_\sigma^{\kappa_\sigma}$ denotes the module formed from $H$ by taking the same composition of powers of tensor, exterior, and symmetric products used to form $\uo^\kappa$.
Now, we can compose the differential operators $D_\tau$.  

\begin{lem}\label{tau-commute}
The differential operators $D_\tau$ commute, i.e. $D_\tau D_\sigma=D_\sigma D_\tau$ for all $\tau, \sigma\in\T_{\realfield}$.
\end{lem}
\begin{proof}
Similar to the proof of \cite[Lemma (2.1.14)]{kaCM}, which reduces the problem to working over $\IC$. 
\end{proof}

We also denote by \index{$D$}$D$ the sum of the differential operators $D_\tau$, i.e.
\begin{align*}
D := (\id\otimes\iota)\circ \nabla,
\end{align*}
where $\iota$ is the sum of the inclusions $\iota_\tau$, i.e. $\iota$ is the inclusion
\begin{align*}
\uo^2\hookrightarrow H^2
\end{align*}
induced by $\uo\hookrightarrow H$.  In analogue with the conventions above, this can also be extended, via the Leibniz rule (i.e. product rule), to a map on $\uo^\kappa$, which also denote by $D$ or \index{$D_\kappa$}$D_\kappa$.

\subsection{$p$-adic differential operators over the $\mu$-ordinary locus}\label{padicdiff_sec}
We recall the construction of $p$-adic differential operators over the $\mu$-ordinary locus $\shmuord$ over $\Witt$, from \cite[Sections 5 and 6]{EiMa}.

\subsubsection{A canonical complement to $\uo$ over the $\mu$-ordinary locus}
We recall the existence of a crucial submodule \index{$U$}$U$ of $\hdrone:=\hdrone\left(\Auniv/\shmuord\right)$. In the following, $\omega:=\omega_{\CA/\shmuord}$, and \index{${\rm Fr}^*$}${\rm Fr}^*$ denotes the Frobenius morphism acting on $\hdrone$.

\begin{prop}[Proposition 5.2.1 of \cite{EiMa}]\label{Usubmodule-section}
There exists a unique submodule $U$ of $\hdrone$ such that
\begin{enumerate}
\item{$U$ is $({\rm Fr}^*)^\be$-stable, where \index{$\be$}$\be =\lcm_{\mathfrak{o}\in\mathfrak{O}_F}  \left(\#\co\right)$.}  
\item{$U$ is $\nabla$-horizontal, i.e. $\nabla(U)\subseteq U\otimes\Omega_{\shmuord/\Witt}^1$.}\label{Prop2}
\item{$U$ is a complement to $\uo$, i.e. $\hdrone = \uo\oplus U$.}
\end{enumerate}
\end{prop} 
When the ordinary locus is nonempty, $U$ is the unit root submodule of $\hdrone$.

\subsubsection{Construction of $p$-adic differential operators}\label{defD_sec}
We now recall the construction of the $p$-adic differential operators over the $\mu$-ordinary locus from \cite[Section 6.2]{EiMa}.  Denote by $\pi_U$ the projection\index{$\pi_U$}
\begin{align*}
\pi_U: \hdrone\twoheadrightarrow \uo
\end{align*}
modulo the module $U$ from Proposition \ref{Usubmodule-section}.  This induces projections\index{$\pi_\tau$} 
\begin{align}
\pi_\tau:\left(\hdrone\right)_\tau\twoheadrightarrow\uo_\tau\label{pitau}\\
\pi_\tau^{\pm}:\left(\hdrone\right)^{\pm}_\tau\twoheadrightarrow\uo_\tau^{\pm}\nonumber
\end{align}
(in the notation of Section \ref{signature-intro}, $\mod U_\tau$ and $U_\tau^{\pm}$, respectively).
For each  weight $\kappa$, we define\index{$\diffop_\tau$}\index{$\diffop_{\kappa, \tau}$}
\begin{align*}
\diffop_\tau:=\diffop_{\kappa,\tau}:=\pi_U\circ D_\tau: \uo^\kappa\rightarrow \uo^\kappa\otimes \uo_\tau^2,
\end{align*}
where $\pi_U$
on $D_\tau(\omega^\kappa)$ is defined by applying $\pi_U$ to each factor $\hdrone$.  Property \eqref{Prop2} of Proposition \ref{Usubmodule-section} guarantees that
$\pi_U\circ \left(D_{\tau_1}\circ \cdots \circ D_{\tau_e}\right) = \diffop_{\tau_1}\circ \cdots \circ\diffop_{\tau_e}$
for any places $\tau_i\in\T_{\realfield}$. 
For each nonnegative integer $e$ and each $\tau\in\T_{\realfield}$, define an operator
\begin{align*}
\diffop_\tau^e :=\diffop_{\kappa, \tau}^e:= \underbrace{\diffop_{\kappa, \tau}\circ\cdots\circ\diffop_{\kappa, \tau}}_{e \mbox{ times}}.
\end{align*}
\begin{prop}\label{prop-sym}
For each nonnegative integer $e$ and each $\tau\in\T_{\realfield}$, the image of $\diffop_{\kappa, \tau}^e$ lies in $\uo^\kappa\otimes\Sym^e\uo_\tau^2$.
\end{prop}

Proposition \ref{prop-sym} on the image of $p$-adic differential operators seems to be accepted in the field but not justified anywhere in the literature.  
As  Proposition \ref{prop-sym} 
is not an immediate consequence of the definition of $\diffop_\tau^e$, we briefly justify it below.  Note that this statement is significant, because it introduces constraints on the amounts by which the differential operators $\diffop_\tau^e$ can raise weights, which are independent of additional constraints that will be forced when the ordinary locus is empty. 

\begin{proof}[Proof of Proposition \ref{prop-sym}]

Note that it is sufficient to prove the statement for sections of $\uo$ defined over a ring $R$ that is dense in the
base 
ring 
$\Witt$  
over which $\shmuord$ is defined.
Going forward, we  take 
$R=\bar{\mathbb{Q}}\cap\Witt$. 
Then it suffices 
(by the density of $R$ in $\Witt$) 
to prove $\diffop^e_\tau(f)\in \uo^\kappa\otimes\Sym^e\uo_\tau^2$ for each global section $f\in H^0\left(\Sh_{/R}, \uo^\kappa\right)$.

By Serre--Tate theory ( \cite[Theorem 6.5]{ShZh}, also \cite[Proposition 2.3.12 (i)]{moonen} and originally due to Serre and Tate in the ordinary case), any $\mu$-ordinary point defined over a finite field $ x\in \shpmuord({\mathbb{ F}})$ admits a (canonical) CM lift  $\tilde{x}\in \Sh(\Witt)$.  Furthermore,   CM points are dense in  the formal neighborhood of $\Sh$ at $x$ (\cite[Theorem 1.1]{ShZh}). 
Hence, it suffices to prove the statement holds locally at each $\mu$-ordinary CM point defined over $R$ (which can, by extending scalars, be viewed as a CM point over $\Witt$).

Fixing an embedding $R\hookrightarrow\IC$ and extending scalars, we may view each automorphic form defined over $R$ as an automorphic form over $\IC$.
The differential operators $\diffop_{\kappa, \tau}^e$ are $p$-adic analogues of the $\ci$ Maass--Shimura operators $\shimuraop_{\rho_\kappa, \tau}^e$\index{$\shimuraop_{\rho_\kappa, \tau}^e$} defined in, e.g., \cite[Section 12.9]{shar}.  More precisely, the $\ci$ Maass--Shimura operators $\shimuraop_{\rho_\kappa, \tau}^e$ can be constructed algebro-geometrically over $\IC$ similarly to the $p$-adic operators $\diffop_{\kappa, \tau}^e$ by replacing $\shmuord$ with $\Sh(\IC)$ and replacing the complement $U$ of $\uo\subset\hdrone$ from the $p$-adic setting with the anti-holomorphic forms $H^{0, 1}\subseteq \hdrone$ in the $\ci$-setting over $\IC$, as in \cite[Chapter II]{kaCM}, \cite[Section 4]{hasv}, \cite[Section 8]{EDiffOps}, and \cite[Section 3.3.1]{EFMV}.  

By the approach first introduced by Katz in \cite[Sections 2.6 and 5.1]{kaCM}, for each form $f$ defined over $R$, the values of $\diffop_{\kappa, \tau}^e(f)$ and $\shimuraop_{\rho_\kappa, \tau}^e(f)$ agree (up to periods) at each $\mu$-ordinary CM point defined over $R$.
 
Thus, it is sufficient to prove the image of the $\ci$-differential operator $\shimuraop_{\rho_\kappa, \tau}^e$ lies in the symmetric product, which Shimura 
accomplished in \cite[Sections 13.1 through 13.8]{shar}.
\end{proof}

\begin{rmk}
One could also  prove Proposition \ref{prop-sym} directly over a $p$-adic ring, by proving $p$-adic analogues of the results on $\ci$-vector fields and complex K\"ahler manifolds Shimura employs in \cite[Section 13]{shar} to prove the image of his operators lies in the symmetric product.  When the ordinary locus is nonempty, this strategy is carried out in  \cite[Remark 5.2.5]{EFMV} via explicit computations of Serre--Tate expansions. When the ordinary locus is empty, analogous computations on Serre--Tate expansion still hold by \cite[Proposition 6.2.5]{EiMa}. 
We know no benefit, however, to carrying out this tedious exercise, since it requires a longer proof and ultimately results in the same conclusion.  Further, we note that there is a well-established benefit to exploiting $\IC$ to prove statements not over $\IC$, e.g. in \cite{kaCM}, as recalled in the proof of Lemma \ref{tau-commute} above.
\end{rmk}

Now, for any symmetric weight $\lambda$ of $\Levi_\tau$, admissible of depth $e$ (at $\tau$), we define\index{$\diffop_\tau^\lambda$}
\begin{align}\label{Dtldefn}
\diffop_\tau^\lambda:=\diffop_{\kappa, \tau}^\lambda:={\rm pr}_{\kappa, \lambda}\circ\left(\id\otimes {\rm pr}_\lambda\right)\circ \diffop_\tau^e: \uo^\kappa\rightarrow\uo ^{\kappa+\lambda},
\end{align}
where ${\rm pr}_\lambda$ denotes projection onto the automorphic sheaf of weight $\lambda$ (inside $\Sym^e\uo_\tau^2$), and ${\rm pr}_{\kappa, \lambda}$ denotes the canonical projection $\uo^\kappa\otimes\uo^\lambda\twoheadrightarrow\uo^{\kappa+\lambda}$ (see Section \ref{weights_sec}).

\subsubsection{Differential operators on OMOL sheaves}\label{padicdiffop}
In \cite[Section 6.3]{EiMa}, we observe that, for any symmetric weight $\lambda$ of $\Levi$,  the $p$-adic Maass--Shimura operators $\CD^\lambda$ on the restriction to $\shmuord$ of automorphic sheaves preserve\index{$\underline{\CD}^\lambda$} 
the standard filtrations, hence inducing operators 
\[\underline{\CD}^\lambda:= ({\rm id}\otimes \varpi^\lambda)\circ\gr(\CD^\lambda):\gr(\uo^\kappa)\to\gr(\uo^\kappa)\otimes \uo^\lambda\to \gr(\uo^\kappa)\otimes \underline{\uo}^\lambda.\]
Via the isomorphisms $\iota_\kappa$, 
the operators
$\underline{\CD}^\lambda$ induce differential operators on OMOL sheaves
\[ \underline{\uo}^{\kappa'}\to \underline{\uo}^{\kappa'}\otimes \underline{\uo}^\lambda\to\underline{\uo}^{\kappa'+\lambda}.\]
For any  weight $\kappa'$ of $\Levii$, by abuse of notation, we still denote them as\index{$\underline{\CD}^\lambda_{\kappa'}$} 
\begin{align*}
\underline{\CD}^\lambda:=\underline{\CD}^\lambda_{\kappa'}:\underline{\uo}^{\kappa'}\to\underline{\uo}^{\kappa'+\lambda}.
\end{align*}

 \newcommand{\modp}{{$\bmod \,p$ }}

\subsection{A first look at differential operators modulo $p$: the OMOL setting}\label{frobsplit_sec}
In Sections \ref{AC-section} and \ref{modpdiff_sec}, we study the $\mod p$ reduction of $p$-adic differential operators on automorphic forms and on OMOL sheaves, respectively, and we analytically continue them beyond the $\mu$-ordinary locus to the entire Shimura variety.
Both results are achieved only under certain restrictions on the weights.  For those weights which satisfy all the assumptions, we compare the two constructions in Proposition \ref{compare_prop}.

Right now, without imposing any restrictions on weights, 
we conclude our introduction to behavior over the $\mu$-ordinary locus by
explaining the relationship between the \modp reductions of the OMOL sheaves $\underline{\uo}^{\kappa'}$ and differential operators $\underline{\CD}^\lambda$ (introduced in Sections \ref{omol-bkgd} and \ref{padicdiffop}, respectively) and the \modp automorphic sheaves $\uo^{\kappa'}$ and differential operators $\CD^\lambda$ studied above.

For  any $\co\in\mathfrak{O}_F$, and $e_\co=\#\co$, write $\CG_\co^\pe:=({\rm Fr}^{e_\co})^*\CA[{\mathfrak p}_\co]$. 
Similarly, for $\be=\lcm_{\co\in\mathfrak{O}}(e_\co)$, write $\CA^{(p^\be)}[p]:=({\rm Fr}^\be)^*\CA[p]$ over $\Shp$. 
By \cite[Lemma 8]{Man05},  over $\shpmuord$, the  filtration of $\CA^{(p^\be)}[p]$  (resp. $\CG_\co^{(p^e)}[p]$, for all $\co\subseteq \CT_F$) induced by the slope filtration is canonically split. 
More precisely, we have the following result.

\begin{lem}( \cite[Lemma 8]{Man05})\label{split_lemma}
Maintaining the above notation, over $\shpmuord$,
there are (compatible) canonical  isomorphisms  of $\mathfrak{D}$-enriched truncated \BT groups of level 1 over $\shpmuord$ \[\gr(\CA)^{(p^e)}[p]\simeq \CA^{(p^e)}[p],\text{ and } \gr(\CG_\co)^{(p^e)}[p]\simeq \CG_\co^{(p^e)}[p], \text{ for all }\co\in\mathfrak{O}_F.\]
Hence, in particular, there are (compatible) canonical isomorphism of $\CO_\shpmuord \otimes_\Witt\CO_F$-modules \[\gr(\uo)^{(p^e)}\simeq\uo^{(p^e)},
 \text{ and } \gr(\uo_\co)^{(p^e)}\simeq \uo_\co^{(p^e)}, \text{ for all }\co\in\mathfrak{O}_F.\]

\end{lem}
Combined with Proposition \ref{filtration_prop}, the above lemma implies the following result.

\begin{prop}\label{split_prop}
For any  weight $\kappa$ of $\Levi$, we have a canonical isomorphism of  $\CO_\shpmuord \otimes_\Witt\CO_F$-modules 
\[(\omega^\kappa)^\pee\simeq \bigoplus_{\kappa'\in\mathfrak{M}_\kappa }(\underline{\uo}^{\kappa'})^\pee.\]
\end{prop}

 At every geometric point of $\shpmuord$, the above isomorphisms agree with those constructed in \cite[Proposition  4.3.3]{EiMa}.
Hence, by combining the above result with \cite[Propositions 6.2.3 and 6.2.5]{EiMa}, we deduce Proposition \ref{split_prop}.

\begin{prop}\label{split_prop}
For any  weight $\kappa$, and any symmetric weight $\lambda$  of $\Levi$, under the identification of  $\CO_\shpmuord \otimes_\Witt\CO_F$-modules 
$(\omega^\kappa)^\pee\simeq \bigoplus_{\kappa'\in\mathfrak{M}_\kappa }(\underline{\uo}^{\kappa'})^\pee$
we have    
\[({{\CD}_\kappa^\lambda})^\pee=\bigoplus_{\kappa',\lambda'}(\underline{\CD}_{\kappa'}^{\lambda'})^\pee,\]
where $\kappa'\in\mathfrak{M}_\kappa,\lambda'\in\mathfrak{M}_{\lambda}$ satisfy $\kappa'+\lambda'\in\mathfrak{M}_{\kappa+\lambda}$.
\end{prop}

For $\lambda$ any  symmetric weight $\lambda$ of $\Levi$, the $p$-adic differential operators $\underline{\CD}^\lambda$ on OMOL sheaves over $\shmuord$ were constructed in \cite[Section 6.3]{EiMa} (as recalled in Section \ref{padicdiffop} of the present paper).
By Proposition \ref{split_prop}, the same definition yields \modp differential operators on the pullback by ${{\rm Fr}^*}^\pee$ of OMOL sheaves over $\shpmuord$, for all symmetric weights $\lambda'$ of $\Levii$. 
 In Proposition \ref{split_prop}, we denoted these operators by $(\underline{\CD}^{\lambda'})^\pee$. Note that they do not arise as the \modp reduction of $p$-adic differential operators.

\section{The Hodge--de Rham filtration in characteristic $p$}\label{HD-section}

\newcommand{\hcrysone}{H_{\rm crys}^1}

\newcommand{\hcrys}{H_{\rm crys}}

\newcommand{\hcrystau}{H_{{\rm crys},\tau}}

\newcommand{\hcrystauu}{H_{{\rm crys},\tau^*}}

A key ingredient in the construction of $p$-adic differential operators over the $\mu$-ordinary locus (as in \cite{EiMa}) is the existence of a (canonical) splitting of the Hodge--de Rham filtration of the universal abelian scheme (Proposition \ref{Usubmodule-section}).  A natural first step towards extending the \modp reduction of these differential operators from the $\mu$-ordinary locus $\shpmuord$ to the whole Shimura variety $\Shp$ is to investigate whether  such a splitting extends from $\shpmuord$ to $\Shp$.
When the ordinary locus is nonempty, such a splitting over $\Shp$ exists.  Indeed, it can be constructed via the conjugate Hodge--de Rham spectral sequence in positive characteristic.  This is also the key ingredient underlying the construction of the ordinary Hasse invariant (see \cite[\S 3.3.1]{EFGMM}).
When the ordinary locus is empty, though, this approach fails.  Instead, we adapt our approach to  the construction of the $\mu$-ordinary Hasse invariant by Goldring--Nicole for PEL-type Shimura varieties in \cite{GN}.  (The yet more general setting of Hodge type should be straightforward using \cite{KWhasse}, once one sorts out numerical details.)  The techniques we develop in this more general setting work only under certain conditions on the weights, which arise from Goldring--Nicole's approach to partial Hasse invariants in \cite{GN}.

For any $\co\in\mathfrak{O}_F$, we write\index{$\cf(\co)$}\index{$\cf(\co)_{>0}$}\index{$\cf(\co)_{<n}$} 
\begin{align*}
\cf(\co)&:=\{\cf(\tau)|\tau\in\co\}\\
\cf(\co)_{>0}&:=\{\cf(\tau)|\tau\in\co \text{ satisfying } \cf(\tau)> 0\}\\
\cf(\co)_{<n}&:=\{\cf(\tau)|\tau\in\co \text{ satisfying } \cf(\tau)<n\}.
\end{align*}

\begin{defi}\label{goodweight}
We call a positive dominant weight $\kappa=(\kappa_\tau)_{\tau\in\T_F}$ {\em good} (for the prime $p$) if 
 $\kappa_\tau$ is a scalar weight of $\GL_{\cf(\tau)}$
whenever
 $\cf(\tau)\neq\min(\cf(\co)_{>0})$.
\end{defi}
In particular, all scalar weights are good.  As we shall in Theorem \ref{PEP_thm}, if a weight is good, then after reducing mod $p$, we can extend the splitting from Proposition \ref{Usubmodule-section} to the entire mod $p$ Shimura variety $\Shp$.

By definition, if we decompose $\kappa$ over $\T_F$, as a product\footnote{Whether we view the decomposition as a product or sum depends on whether we are considering weights as characters or as the corresponding tuples of integers.} of weights $\kappa_\tau$ supported at  $\tau\in\T_F$, then $\kappa$ is good if and only if the weights $\kappa_\tau$ are good for all $\tau\in\T_F$.

The main goal of this section is to establish the following result, whose proof relies on the material introduced in the remainder of this section.
Below, $\HA_\Sigma\in H^0(\Shp,\omega_{\CA/\Shp}^{\kappa_{\ha,\Sigma}})$  denotes the $\Sigma$-Hasse invariant, and  $\HA=\HA_{\T_F}$ (see Section \ref{hasse_sec}).
We refer to Definition \ref{twistweight} for the notation $||\kappa||\in \ZZ^{|\T_F|}$.

\begin{thm}\label{PEP_thm}
Let $\Sigma\subseteq\T_{F}$, and let $\kappa$ be a weight. Assume $\kappa$ is good.  Then each of the following statements holds:

\begin{enumerate}
\item There exists
a morphism of $\CO_\Shp$-modules\index{$\Pi^\kappa$}
\begin{align}
\Pi^\kappa:  \hdrone(\CA/\Shp)^\kappa \to \omega_{\CA/\Shp}^{\kappa+||\kappa||\kappa_\ha}=
\omega_{\CA/\Shp}^\kappa\otimes\omega_{\CA/\Shp}^{||\kappa||\kappa_\ha}\label{PEP_thm1}
\end{align}
that satisfies the equality
 $\Pi^\kappa_{\vert \shpmuord}=\HA^{||\kappa||}\cdot \overline{\pi}_U,$
where $\overline{\pi}_U$ is the \modp reduction of the map
$\pi_U:  \hdrone(\CA/\shpmuord)^\kappa \to\omega_{\CA/\shmuord}^\kappa$
 induced via the $\kappa$-Schur functor by 
 the projection 
 $\pi_U:\hdrone(\CA/\shpmuord)\to\omega_{\CA/\shpmuord}$ defined in Section \ref{defD_sec}.

\item Assume $\kappa$ is supported at $\Sigma$.  Then there exists
a morphism of $\CO_\Shp$-modules
\[\Pi_\kappa: 
D(\omega_{\CA/\Shp}^\kappa)
\to \omega_{\CA/\Shp}^{\kappa+\kappa_{\ha,\Sigma}}\otimes \omega^2_{\CA/\Shp}=\omega_{\CA/\Shp}^\kappa\otimes\omega_{\CA/\Shp}^{\kappa_{\ha,\Sigma}}\otimes \omega^2_{\CA/\Shp}\] 
which satisfies the equality
 $\Pi_{\kappa\vert \shpmuord}=\HA_\Sigma\cdot \overline{\pi}_U,$
where $\overline{\pi}_U$ is the \modp reduction of the projection
$\pi_U: D(\omega^\kappa_{\CA/\shpmuord})
\to\omega_{\CA/\shmuord}^\kappa\otimes  \omega_{\CA/\shpmuord}^2$
 defined in Section \ref{defD_sec}.

\end{enumerate}
\end{thm}

\begin{rmk}
Note that since the Hasse invariant is of scalar weight, we do indeed have an equality
 $\omega_{\CA/\Shp}^{\kappa+||\kappa||\kappa_\ha}=
\omega_{\CA/\Shp}^\kappa\otimes\omega_{\CA/\Shp}^{||\kappa||\kappa_\ha}$ as in Equation \eqref{PEP_thm1}.
\end{rmk}

\begin{proof}[Proof of Theorem \ref{PEP_thm}]
By decomposing the weight $\kappa$ as a product of weights supported at single $\tau\in\T_F$, as $\tau$ varies in $\T_F$, we reduce the proof of Theorem \ref{PEP_thm}
to the special case of a good weight $\kappa$ which is supported at a single $\tau$. 
For scalar weights, the result follows from Proposition \ref{PIEpi1} which relies on the construction of the $\mu$-ordinary Hasse invariant (\cite[Lemmas 3.1 and 3.2]{GN}, see Lemma \ref{lemma1}). For non-scalar weights (which, by definition of good weights, are only supported at $\tau\in\T_F$ satisfying $\cf(\tau)=\min(\cf(\co)_{>0})$), the result follows from Proposition \ref{PIEpi2} which relies on Lemma \ref{lemma2}. 
\end{proof}

In this section, we work in positive characteristic over $\Shp$. Set $\hdrone:=\hdrone(\CA/\Shp)$, and $\omega:=\omega_{\CA/\Shp}$.
We write $\hcrysone:=\hcrysone (\CA/\Shp)$ for the Dieudonn\'e crystal of $\CA/\Shp$,
 and identify its \modp reduction with $\hdrone$. In general, we denote by $\overline{(\cdot)} $ the reduction \modp of an object  over $\Witt$.

For convenience, through out this section, we set $\T=\T_F$. Also, for any $\tau\in\T_F$, write  $\tau_0=\tau_{\vert F_0}\in\T_{F_0}$, and following the conventions of Section \ref{signature-intro}, $(\cdot)_\tau=(\cdot )^+_{\tau_0}$ if $\tau\in \Sigma_F$ and
$(\cdot)_\tau=(\cdot )^-_{\tau_0}$ if $\tau\not\in \Sigma_F$.

Fix $\tau\in\T$. The restriction to $\hcrystau^1$ of Frobenius ${\rm Fr}^*$ on $\hcrysone$ induces  a map\index{${\rm Fr}^*_\tau$}
\[{\rm Fr}^*_\tau:={\rm Fr}^*|_{\hcrystau^1}: \hcrystau^1\rightarrow H^1_{\mathrm{crys}, \tau\circ\sigma}.\]

For  $e=e_\tau := \#\orbit_\tau$,
we have  $ \tau\circ\sigma^e = \tau$, and $({ \rm Fr}^*_\tau)^e:(\hcrystau^1)^{(p^e)}\to\hcrystau^1.$
We write\index{$\phi_\tau$}
$\phi_\tau:=({ \rm Fr}^*_\tau)^e$.

\subsection{The scalar-weight case}\label{scalarsplit}
Without loss of generality, we assume $\cf(\tau)\neq 0$.  (When $\cf(\tau)=0$, the Hodge--de Rham filtration is trivial at $\tau$ and all automorphic weights supported at 
$\tau$ are trivial.)
Then the Hodge--de Rham filtration on $\hdrtauu^1$, $\omega_\tauu\subseteq \hdrtauu^1$, induces a canonical filtration on  $\bigwedge^{\cf\left(\tau^*\right)}\hdrtauu^1=\hdrtauu^{\cf\left(\tau^*\right)}$.  The first step in the filtration is the locally free subsheaf\index{${\mathcal W}_\tauu$}
\[{\mathcal W}_\tauu:=\mathrm{Fil}^1\left(\hdrtauu^{\cf(\tau)}\right)=
\ker\left( \bigwedge^{\cf(\tau)}\hdrtauu^1\to \bigwedge^{\cf(\tau)}\left(\hdrtauu^1/\omega_\tauu\right)\right) .
\]

\begin{lem}(\cite[Lemmas 3.1 and 3.2]{GN})\label{lemma1}
For $\tau\in\T$ with $\cf(\tau)\neq 0$, define 
\[\Phi_\tauu:= \phi_\tauu\wedge\cdots \wedge\phi_\tauu: \bigwedge^{\cf(\tau)}\hcrystauu^{1\, (p^e)}\to \bigwedge^{\cf(\tau)}\hcrystauu^1 \]
and  $c_\tauu :=\sum_{\tau'\in\orbit_\tau, \\
\cf(\tau')>\cf(\tau)} \left(\cf(\tau')-\cf(\tau)\right).$
Then each of the following statements holds:
\begin{enumerate}
\item $\Phi_\tauu$  is divisible by $p^{c_\tauu}$.
 \item $\overline{\Phi_\tauu/p^{c_\tauu}}$ vanishes on ${\mathcal W}^{(p^e)}_\tauu$. 
\end{enumerate}
\end{lem}

\begin{remark}\label{ctau}
Set
$a_i=a_i^\tau:=\vert\{\tau'\in\orbit_\tau|\cf(\tau')>n-i\}\vert
,$ for $i=1,\dots, n$.  By definition (Definition \ref{muNP}), the rational numbers
$a^\tau_1/e\leq \cdots\leq a^\tau_n/e$,  
are the slopes of the Newton polygon $\nu_{\co_\tau}(n,\cf)$ (occurring with multiplicity).
Then
\[c_\tauu= \sum_{i=1}^{f\left(\tau\right)}a_i^\tauu.\]
\end{remark}

We write
$\varphi_\tauu: \left(\bigwedge^{\cf(\tau)}\hdrtauu^{1\, (p^e)}\right)/{\mathcal W}^{(p^e)}_\tauu\to \bigwedge^{\cf(\tau)}\hdrtauu^{1}$
for the morphism induced by $\overline{\Phi_\tauu/p^{c_\tauu}}$, and denote by
$\varphi^0_\tauu:  \left(\bigwedge^{\cf(\tau)}\hdrtauu^{1\,(p^e)}\right)/{\mathcal W}_\tauu^{(p^e)}  \to \left(\bigwedge^{\cf(\tau)}\hdrtauu^{1}\right)/{\mathcal W}_\tauu $ its composition with the projection modulo ${\mathcal W}_\tauu$. 
They fit in the commutative diagram

$$\xymatrix{
\bigwedge^{\cf(\tau)}\hdrtauu^{1\, (p^e)} \ar@{->>}[dd]^{\bmod{\mathcal W}^{(p^e)}_\tauu}\ \ar[rrrr]^{\overline{\Phi_\tauu/p^{c_\tauu}}} &&&&\bigwedge^{\cf(\tau)}\hdrtauu^1\ar@{->>}[dd]^{\bmod {\mathcal W}_\tauu}\\
\\
\left(\bigwedge^{\cf(\tau)}\hdrtauu^{1\, (p^e)}\right)/{\mathcal W}^{(p^e)}_\tauu \ar[rrrr]^{\varphi^0_\tauu}\ar[rrrruu]^{\varphi_\tauu}     &&&     &\left(\bigwedge^{\cf(\tau)}\hdrtauu^1\right)/{{\mathcal W}_\tauu}
}$$

The projection
$\bigwedge^{\cf(\tau)}{\rm pr}_\tauu:\bigwedge^{\cf(\tau)} \hdrtauu^1\to \bigwedge^{\cf(\tau)}(\omega^\vee)_\tauu$ induces an isomorphism 
\[\left(\bigwedge^{\cf(\tau)}\hdrtauu^1\right)/{\mathcal W}_\tauu \simeq \bigwedge^{\cf(\tau)}(\omega^\vee)_\tauu=
\bigwedge^{\cf(\tau)}(\omega_{\tau})^\vee.\]
We define $\Pi_\tau:=(\varphi_{\tau^*})^\vee$ and $h_\tau:=(\varphi^0_{\tau^*})^\vee$, and  consider the commutative diagram (dual to the one above)

$$\xymatrix{
\bigwedge^{\cf(\tau)}\hdrtau^{1\,(p^e)} &&&&\bigwedge^{\cf(\tau)}\hdrtau^1\ar[llll]_{ (\overline{\Phi_{\tau^*}/p^{c_{\tau^*}}})^\vee}\ar[lllldd]_{\Pi_\tau}\\
\\
\bigwedge^{\cf(\tau)}\omega_{\tau}^{\left(p^e\right)}\ar@{^{(}->}[uu]    
&&&     &\bigwedge^{\cf(\tau)}\omega_{\tau}\ar[llll]_{h_\tau}\ar@{^{(}->}[uu] \\
}$$

\begin{prop}\label{PIEpi1} Maintaining the above notation,
for each $\tau\in\CT$ with $\cf(\tau)\neq 0$, the morphism of $\CO_\Shp$-modules
\[\Pi_{\tau}:  \bigwedge^{\cf(\tau)} \hdrtau^1\to  \bigwedge^{\cf(\tau)}\omega_\tau^{(p^e)}= \bigwedge^{\cf(\tau)}\omega_\tau \otimes |\omega_\tau|^{p^{e_\tau-1}}\] 
satisfies the equality
 \[\Pi_{\tau}|_\shpmuord=\HA_\tau\cdot \left (\bigwedge^{\cf(\tau)}\overline{\pi}_{\tau}\right).\]
where 
$\HA_\tau \in H^0(\Shp,|\omega_\tau|^{p^{e_\tau}-1})$ is the $\tau$-Hasse invariant,  
and $\overline{\pi}_\tau$ is the \modp reduction of the morphism $\pi_\tau:\hdrone(\CA/\shmuord)_\tau^1\to(\omega_{\CA/\shmuord})_\tau$ given in (\ref{pitau}). 
\end{prop}

\begin{proof}
By definition, 
for each $\tau\in\T$ with $\cf(\tau)\neq 0$, the $\tau$-Hasse invariant $\HA_\tau$ 
satisfies
\[
h_\tau= 1\otimes \HA_\tau:  |\omega_\tau| :=\bigwedge^{\cf\left(\tau\right)}\omega_{\tau} \longrightarrow
{\bigwedge^{\cf\left(\tau\right)} {\omega_{\tau}^{\left(p^{e_\tau}\right)}}}
=|\omega_\tau|^{p^{e_\tau}}
=
 |\omega_\tau|\otimes|\omega_\tau|^{p^{e_\tau}-1}.\] 

Hence, the statement is an immediate consequence of the construction.
\end{proof}

\subsection{The vector weight case}\label{vwcase-section} 
We assume $\cf(\tau)=\min\left(\cf(\co_\tau)_{>0}\right)$. 
The following result is a variation of Lemma \ref{lemma1}

\begin{lem}\label{low}\label{lemma2}
Let  $\tau\in\T$ satisfying $\cf(\tau)=\min\left(\cf(\co_\tau)_{>0}\right)$. 
Consider the map 
\[\phi_\tauu:\hcrystauu^{1\,(p^e)}\to\hcrystauu^1,\]

and let  $a_\tauu:=
\vert\{\tau'\in\orbit_\tauu|\cf(\tau')=n\}\vert$.
\begin{enumerate}
\item $\phi_\tauu$  is divisible by $p^{a_\tauu}$; 
\item $\overline{\phi_\tauu/p^{a_\tauu}}$ vanishes on $\omega_\tauu$. 
\end{enumerate}
\end{lem}
\begin{proof}
By Remark \ref{ctau}, $a_\tau=a_1^\tau$ and 
$c_\tauu=\cf(\tau)a_\tauu$, since, by assumption,  $\cf(\tauu)=\max\left(\cf(\co_\tauu)_{<n}\right)$. This observation suffices to adapt  the arguments in
\cite[Lemmas 3.1 and 3.2]{GN} to establish the above statements.
\end{proof}

We define $\tilde{\Pi}_\tau$ as the morphism dual to the map $\hdrtauu^{1\, (p^e)}/\omega^{(p^e)}_\tauu\to \hdrtauu^{1}$ induced by 
 by $\overline{\phi_\tauu/p^{a_\tauu}}$, and denote by
$\tilde{h}_\tau$
  its composition with the inclusion  $\omega_\tau\hookrightarrow \hdrtau^1$.  They fit in the commutative diagram

$$\xymatrix{
\hdrtau^{1\,(p^e)} &&&&\hdrtau^1\ar[llll]_{ (\phi_{\tau^*}/p^{a_{\tau^*}})^\vee}\ar[lllldd]_{\bar{\Pi}_\tau}\\
\\
\omega_{\tau}^{\left(p^e\right)}\ar@{^{(}->}[uu] 
&&&     &\omega_{\tau}\ar[llll]_{\bar{h}_\tau}\ar@{^{(}->}[uu] \\
}$$

It follows from the construction that $\bigwedge^{\cf(\tau)} \tilde{\Pi}_\tau$ and $\bigwedge^{\cf(\tau)} \tilde{h}_\tau$ agree respectively with the morphisms $\Pi_\tau$ and $h_\tau$ defined in Section \ref{scalarsplit}. 
In particular, let  \[{^*\tilde{h}}_{\tau}: \omega^{(p^e)}_\tau\to\omega_\tau\otimes |\omega^{(p^e)}|\otimes |\omega_\tau|^{-1}= \omega_\tau\otimes |\omega_\tau|^{p^e-1}\] denote the adjugate of $\tilde{h}_\tau$ (see \cite[\S 3.3.1]{EFGMM}), then it satisfies
$^*\tilde{h}_\tau\circ \tilde{h}_\tau=1\otimes \HA_\tau$,
where $\HA_\tau$ is the $\tau$-Hasse invariant.

\begin{prop}\label{PIEpi2} Maintaining the above notation,
for each $\tau\in\CT$  satisfying $\cf(\tau)=\min\left(\cf(\co_\tau)_{>0}\right)$, 
the morphism of $\CO_\Shp$-modules
\[\Pi_{\tau}:= { ^*\tilde{h}_\tau}\circ\tilde{\Pi}_\tau:   \hdrtau^1\to \omega_\tau \otimes |\omega_\tau|^{p^{e_\tau-1}}\] 
satisfies the equality
 \[\Pi_{\tau}|_\shpmuord=\HA_\tau\cdot\pi_{\tau}.\]

\end{prop}

\begin{proof}
The statement is a consequence of the constructions. 
\end{proof}

By construction, the morphism in Proposition \ref{PIEpi1} agrees with the top exterior power of the morphism  in Proposition \ref{PIEpi2}.

\section{Analytic continuation of the \modp reduction of differential operators}\label{AC-section}

\label{modpdiffop_sec}

In this section,  under some restriction on the weights, we construct weight-raising differential operators $\Theta^\lambda=\Theta_\kappa^\lambda$  on the space of mod $p$ automorphic forms of weight $\kappa$ on $\shp$ which are \modp analogues of Maass--Shimura differential operators. Furthermore, we prove that the restrictions to the $\mu$-ordinary locus $\shpmuord$ of the operators $\Theta^\lambda$ agree with the \modp reduction of the differential operators $\CD^\lambda$ constructed in \cite{EiMa}, multiplied by a power of the $\mu$-ordinary Hasse invariant which\textemdash most importantly---depends only on the weight $\lambda$, and not on $\kappa$.

We prove the following generalization of  \cite[Theorem 3.4.1]{EFGMM} (which, in contrast to the present paper, had required that $p$ split completely in $\reflexfield$).
\begin{thm}\label{ANAlambda_thm}
Let $\Sigma\subseteq \T_F$, 
and let
$\lambda$ be a symmetric weight supported at $\Sigma$. 
Assume either  $\lambda-\delta(\tau)$, for some $\tau\in\Sigma_{F}\cap\Sigma$,  or $\lambda$ is good, and set either $\lambda'=\lambda-\delta(\tau)$ or $\lambda'=\lambda$, respectively.
Then for any good weight $\kappa$ supported at $\Sigma$, there exists a differential operator
\[\Theta^\lambda_\Sigma:=\Theta^\lambda_\kappa:\omega^\kappa\to \omega^{\kappa+\lambda+(|\lambda|/2)\kappa_{\ha_\Sigma} +||\lambda'||\cdot \kappa_\ha}=\omega^{\kappa+\lambda}\otimes\omega^{(|\lambda|/2)\kappa_{\ha,\Sigma}+||\lambda'||\cdot\kappa_\ha}\]
which satisfies 
\[\Theta_\Sigma^\lambda|_\shpmuord\equiv \HA^{{||\lambda'||}}\HA^{|\lambda|/2}_\Sigma\cdot\CD^\lambda \mod p. \]
\end{thm}

\begin{rmk}
The above statement proves that the differential operator $\CD^\lambda$ can be analytically continued from $\shpmuord$ to the whole \modp Shimura variety $\Shp$. 
The explicit power of the Hasse invariant given in Theorem \ref{ANAlambda_thm}  is not optimal, however. For example, in the case when the ordinary locus is not empty, it is higher than the power given for the ordinary case in \cite[Theorem 3.4.1]{EFGMM}. This is due to the limitations of our construction. In the case when all symmetric weights are good, Theorem 
\ref{ANAlambda_thm} can be improved to the expected power of $\HA$ (see Corollary \ref{allgood_coro}).  
\end{rmk}

\begin{remark}\label{scalarsym}

In the symplectic case (C) all weights are good; in the unitary case (A) this is not the case, and goodness  is in general a strong restriction.

The existence of good symmetric weights is also nontrivial, in case (A). 
Non-zero scalar (and hence good) symmetric weights exist if and only if there is $\tau\in\T_F$ satisfying $\cf(\tau)=\cf(\tau^*)$. In particular, 
non-zero parallel symmetric weights exist only in the symplectic and hermitian case.
Indeed, if a scalar weight $\underline{\ell}=({\ell}_\tau)_{\tau\in\T}\in\ZZ^\T$ is symmetric, then
$\ell_\tau=0$ unless  $\cf(\tau)=\cf(\tau^*)$ and $\ell_\tau=\ell_\tauu$. The converse also holds.
Good (non-scalar) symmetric weights occur more generally.
For example, they exist  if there is an orbit $\co$ in $\T_F$ satisfying $\cf(\co)\subseteq\{0, j(\co) ,n\}$, for some $j(\co)\in\ZZ$ (equivalently, such that the polygon $\nu_\co(n,\cf)$ has at most 2 slopes), or 
if there is an orbit $\co$ in $\T_F$ satisfying $\cf(\tau)>n/2$ for all $\tau\in\co$. 
\end{remark}

\subsection{The differential operators $\Theta_\tau$}
The $p$-adic Maass--Shimura operators $\CD^\lambda$ are constructed via iterations starting from the operators $\CD_{\tau}$, for $\tau\in\Sigma_\tau$. We first construct analogous \modp differential operators $\Theta_{\tau}$, for all $\tau\in\Sigma_{F}$.

More precisely, we establish the following special case of Theorem \ref{ANAlambda_thm}.

\begin{thm} \label{ANA_thm}
Let $\Sigma\subseteq\T_F$ and $\tau\in\Sigma_{F}$.
For any good weight $\kappa$ supported at $\Sigma$, there is a differential operator
\[\Theta_{\Sigma,\tau}: \omega^\kappa\to 
\omega^{\kappa}
\otimes\omega^{\kappa_{\ha,\Sigma}} \otimes \omega_{\tau}^{2},\]
which satisfies 
\[\Theta_{\Sigma, \tau}|_\shpmuord \equiv \HA_\Sigma \cdot\CD_{\tau} \mod p.\]
\end{thm}

\begin{rmk}
The reader might be surprised to see that in order to extend $\CD_\tau$ to the whole variety $\Shp$, it is only necessary to multiply by the partial Hasse invariant at places in the support $\Sigma$ of $\kappa$, i.e. for any modular form $f$, the poles of $\CD_\tau(f)$ are determined by places in $\Sigma$, regardless of $\tau$.  This can be seen from the construction of $\diffop_\tau$ in terms of $\nabla$ and the fact that $\nabla(fu) = f\nabla(u) + u\otimes df$ for any regular function $f$ and $u\in \otimes_{\tau'\in\Sigma}\uo_{\tau'}$ (which shows that only $\nabla(u)$ could contribute poles and only at places in $\Sigma$).
\end{rmk}

For convenience, in the following, we denote by\index{$\overline{D}_\tau$}
\[\overline{D}_{\tau}:\omega^\kappa \to (\hdrone)^\kappa\otimes\omega_{\tau}^2\]
the \modp reduction of the algebraic differential operator $D_\tau$ over $\Sh$ defined in Section \ref{defDalg_sec}.

\begin{proof}[Proof of Theorem \ref{ANA_thm}]
 For any good weight $\kappa$ supported at $\Sigma$,  define the differential operator
\[\Theta_{\Sigma, \tau}:=\Theta_{\kappa,\tau}:=(\Pi_\kappa\otimes{\mathbb I})\circ\overline{D}_{\tau}: \omega^\kappa\to \overline{D}_\tau(\omega)\subseteq (\hdrone)^\kappa\otimes\omega_{\tau}^2\to\omega^\kappa\otimes
\omega^{\kappa_{\ha,\Sigma}}\otimes \omega_{\tau}^{2}.
\]

By comparing the  constructions of the differential operators  $\Theta_{\kappa,\tau}$ and  $\CD_{\kappa,\tau}$, we see that the statement  is an immediate consequence of Theorem  \ref{PEP_thm}.
\end{proof}

\begin{coro}\label{allgood_coro}
Let $\Sigma\subseteq \T_F$, and let $\lambda$ be a symmetric weight supported at $\Sigma$. 
Assume all symmetric weights supported at $\Sigma$ are good. 
Then for any good weight $\kappa$ supported at $\Sigma$, there exists a differential operator
\[\Theta_\Sigma^\lambda:=\Theta^\lambda_\kappa: \omega^\kappa\to 
\omega^{\kappa+\lambda}\otimes\omega^{(|\lambda|/2)\kappa_{\ha,\Sigma}},\]
which satisfies 
\[\Theta_\Sigma^\lambda|_ \shpmuord \equiv \HA^{|\lambda|/2}_\Sigma \cdot\CD^\lambda \mod p.\]
\end{coro}

\begin{proof}
When all symmetric weights supported at $\Sigma$ are good, we may  construct the operators $\Theta_\Sigma^\lambda$ by iterating the operators $\Theta_{\Sigma,\tau}$, for $\tau\in\Sigma\cap\Sigma_F$, from Theorem \ref{ANA_thm}.
\end{proof}

Thus, in the setting of \cite{EFGMM}, we recover \cite[Theorem 3.4.1]{EFGMM}.  

\begin{rmk} Given two positive dominant weights $\kappa,\lambda$, if $\kappa$ good, then $\kappa+\lambda$ is good if and only if  $\lambda$ is good. Thus, all symmetric weights supported at $\Sigma$ are good if and only if the weights $\delta(\tau)=\delta_\tau+\delta_\tauu$ are good for all $\tau\in\Sigma\cap\Sigma_F$. 
In particular, in case (A), all weights are good if and only if for every orbit $\co$ in $\T$ there exists an integer $j(\co)\in\ZZ$ such that $\cf(\co)\subseteq\{0,j(\co),n\}.$
For example, all weights are good if the signature is definite at all or all but one real place, and all primes $v|p$ of $F_0$ split in $F$.

Indeed, for any $\tau\in\T$, $\cf(\tau)=\min\left(\cf(\co_\tau)_{>0}\right)$ if and only if $\cf(\tau^*)=\max\left(\cf(\co_\tauu)_{<n}\right)$. Therefore, assuming all weights are good,  if there exists $\tau\in\co$ such that $\cf(\tau)\neq 0,n$ then  $\cf(\co)= \{0,j(\co),n\}$ for $j(\co)=\max\left(\cf(\co)_{<n}\right)=\min\left(\cf(\co)_{>0}\right)$.  The converse also holds.

Note that if $\co=\co^*$ (equivalently, if the prime $v|p$ of $F_0$ is inert in $F$), then $\cf(\co)=\{0,j(\co),n\}$ for some $j(\co)\in\ZZ$ if and only if $n$ is even, and for all $\tau\in\co$ if $\cf(\tau)\neq 0,n$ then $\cf(\tau)=n/2$.

\end{rmk}

\subsection{Proof of Theorem \ref{ANAlambda_thm} on \modp reductions of differential operators}

In general,  the weights $\delta(\tau)$ of the automorphic sheaves $\omega_{\tau}^2$, $\tau\in\Sigma_F$,  are not good for all $\tau\in\Sigma\cap\Sigma_F$, hence the operators $\Theta_{\Sigma, \tau}$ cannot be iterated. Instead, we modify our construction.

\begin{proof}[Proof of Theorem \ref{ANAlambda_thm}]
We treat the two cases $\lambda'=\lambda$ a good weight and $\lambda'=\lambda-\delta(\tau)$ a good weight separately.
We first consider the case where $\lambda'=\lambda $ is good.
Then, the operator $\CD^\lambda$ is constructed by iterating $\CD_\tau$, for $\tau\in\Sigma\cap\Sigma_F$.
Let  
\[\overline{D}^\lambda=\overline{D}^\lambda_\kappa:(\hdrone)^\kappa\to (\hdrone)^\kappa\otimes (\hdrone)^\lambda\]
denote the differential operator obtained by iterating the operators $\overline{D}_{\tau}$.

We define $\Theta_\Sigma^\lambda:=\Theta_\kappa^\lambda:= {\rm pr}_{\kappa,\lambda}\circ (\Pi^{(|\lambda|/2)}_\kappa\otimes\Pi^\lambda)\circ \overline{D}_\kappa^\lambda\circ\iota$,
\[
\omega^\kappa\hookrightarrow
(\hdrone)^\kappa\to (\hdrone)^\kappa\otimes (\hdrone)^\lambda\to
(\omega^\kappa\otimes \omega^{(|\lambda|/2)\kappa_{\ha,\Sigma}})\otimes (\omega^\lambda\otimes \omega^{{||\lambda||}\cdot\kappa_\ha})=\]
\[=(\omega^\kappa\otimes \omega^\lambda)\otimes \omega^{(|\lambda|/2)\kappa_{\ha,\Sigma}+||\lambda||\cdot\kappa_\ha}\to 
\omega^{\kappa+\lambda}\otimes\omega^{(|\lambda|/2)\kappa_{\ha,\Sigma}+||\lambda||\cdot\kappa_\ha},\]
where $\Pi^{(|\lambda|/2)}_\kappa$ denotes the $|\lambda|/2$-times iteration of the map $\Pi_\kappa$.

Then the statement is a consequence 
of Theorem  \ref{PEP_thm}.
Finally, if  $\lambda'=\lambda-\delta(\tau)$ is a good weight,  for some $\tau\in\Sigma$, then $|\lambda|=|\lambda'|+2$ and we define \[\Theta_\Sigma^\lambda:=\Theta_{\Sigma,\tau}\circ \Theta_\Sigma^{\lambda'}.\] 
\end{proof}

\section{New classes of entire mod $p$ differential operators}\label{modpdiff_sec}

\newcommand{\V}{{\mathcal V}}

In this section, 
we construct new entire differential operators $\tth$ over $\Shp$ that, in contrast to the operators in Section \ref{AC-section}, cannot be obtained as (the analytic continuation of) the mod $p$ reduction of $p$-adic Maass--Shimura operators from Section \ref{DiffOpsReview-section}. 
As discussed in Section \ref{AC-section}, good symmetric weights exists only under favorable restrictions on the signature of the Shimura datum; the goal of this section is to construct a new class of weight-raising \modp differential operators,  to which such restrictions do not apply. 
Our starting point is the observation that, for any $\tau\in\CT_F$, the action of Verschiebung  on \modp automorphic forms transfers forms of weight supported at $\tau$ to forms of weight supported at $\tau\circ\sigma^{-1}$.  Hence, by composing the operators $\Theta_\tau$ with appropriate powers of Verschiebung, we obtain \modp differential operators which raise the weight of automorphic forms by  (non-symmetric) good weights,  and can therefore  be iterated and composed without restrictions.

To better explain our construction, we observe that over the $\mu$-ordinary locus $\shpmuord$ twisting by Verschiebung agrees with the \modp reduction of the canonical projection from automorphic sheaves onto their canonical OMOL quotients, that is the OMOL sheaves of the same weights (defined in Proposition \ref{filtration_prop}(3)). Under some restrictions on the weights, we show that the \modp reductions of OMOL sheaves naturally extend to the whole Shimura variety  in positive characteristic, and that  the 
\modp reduction of the differential operators $\underline{\CD}^\lambda$ on them also extends to entire differential operators $\uth^\lambda$. 
The new differential operators $\tth$ are constructed by realizing \modp OMOL sheaves as subsheaves of automorphic sheaves of higher (good) weights.

While our techniques differ, the results in this section generalize some of the results on $\Theta$-operators in \cite[\S4]{DSG}, and \cite[\S 4]{DSG2} (see  Remark \ref{DSG_compare}).

{\bf We assume that there exists $\co\in\mathfrak{O}$ such that $0,n\not\in\cf(\co)$ and $e_\co\geq 2$.}

\subsection{Veschiebung twist}

Let  $\Phi$ denote absolute frobenius on $\Shp$; we write $\CA^{(p)}=\Phi^*\CA$ over $\Shp$, $\nabla=\nabla_{\CA/\Shp}$, and $\nabla^{(p)}=\nabla_{\CA^{(p)}/\Shp}$.
Let \[V:\hcrysone=H^1_{\rm crys}(\CA/\Shp)\to (\hcrysone)^{(p)}=H^1_{\rm crys}(\CA^{(p)}/\Shp)\] denote relative Verschiebung, then 
 $(V\otimes \mathbb{I}) \circ\nabla=\nabla^{(p)}\circ V$. For any $\tau\in\CT_F$, $V$ induces a homomorphism
\[V_\tau:H^1_{{\rm crys},\tau \circ\sigma} \to
(\hcrystau^1)^{(p)}.\]

\begin{lem}\label{Lemma} Let  $\tauco\in\CT_F$ satisfying $\cf(\tauco)=\min(\cf(\co)_{>0})$. For any  $1\leq j\leq  e_\co$.\begin{enumerate}
\item 
$\overline{V^j_\tauco} (
H^1_{{\rm dR},\tauco\circ\sigma^j}
)\subseteq\uo^{(p^j)}_{\tauco}$;
\item  $\overline{\nabla}^{(p^j)}(\uo_\tauco^{(p^j)})\subseteq \uo_\tauco^{(p^j)}\otimes \Omega^1_{\Shp/\overline{\mathbb{F}}_p}$;
\end{enumerate}
\end{lem}

\begin{proof}
For part (1), the statement is equivalent to the vanishing  
of the map \[g\circ\overline{V}^j_{\tauco}:H^1_{{\rm dR},\tauco\circ\sigma^j}\to H_{{\rm dR},\, \tauco}^{1\,(p^j)}\to \left(\hdrone/{\rm Fil}^1(\hdrone)\right)^{(p^j)}_\tauco,\]
where $g$ is the natural projection of $\hdrone$ modulo ${\rm Fil}^1$.
By the density of the $\mu$-ordinary locus, it suffices to prove the vanishing for every $\mu$-ordinary geometric point.

For part (2),  again it suffices to check the statement over the $\mu$-ordinary locus. By the functoriality of the Gauss--Manin connection, we have
\[(\bar{V}^j\otimes \mathbb{I})\circ \overline{\nabla}=\overline{\nabla}^{(p^j)}\circ \bar{V}^j,\]
hence the statement holds at all $\mu$-ordinary geometric points by part (1).
\end{proof}

\subsection{Analytic continuation of \modp OMOL sheaves}
For each $\co\in\mathfrak{O}_F$, we fix  $\tau_\co\in\co$ satisfying $\cf(\tau_\co)=\min(\cf(\co)_{>0})$, and  for any $0\leq j<e_\co$, write \[\V_{j}:=\V_{j, \tau_\co}:=\overline{V}^{j}_{\tau_\co}
:\uo_{\tau\circ \sigma^j}\to\uo^{(p^j)}_{\tau_\co}.\] 
For any $\tau\in\co$, we defined $0\leq j_\tau:=j_{\tau,\tau_\co}<e_\co$ by the equality  $\tau=\tau_\co\circ \sigma^{j_\tau}$.

\begin{remark}

If $\co=\co^*$, then for each $\tau\in\co$,  we have $\tau^*=\tau\circ \sigma^{e/2}\in\co$, for $e=e_\co$, and in particular 
$\tau=\tau_\co\circ\sigma^j$ if and only if $\tau^*=\tau_\co\circ\sigma^{j+e/2}$, for any 
$0\leq j<e$. Hence,  $j_\tauu\equiv j_\tau+e/2\mod e$.
\end{remark}

 \begin{lem}\label{lem}
Let $\co\in\mathfrak{O}$ satisfying $0\not\in\cf(\co)$. For any $1\leq j\leq e_\co$, let $\tau=\tau_0\circ \sigma^j\in\co$,  and consider the restriction to $\shpmuord$ of the homomorphism of $\CO_F\otimes\CO_\Shp$-modules
\[\overline{V}^{j}:(\hdrone)_\tau\to(\hdrone)^{(p^{j})}_{\tau_\co}.\]
Then each of the following statements is true.
\begin{enumerate}
\item\label{lem2}
$\pi_\tauco^{(p^j)} \circ \overline{V}^{j}\vert_{\shpmuord}=\overline{V}^{j}\vert_{\shpmuord}\circ\pi_\tau$.
\item\label{lem1}
The restriction  $\V_{j}:\uo_\tau\to\uo^{(p^j)}_{\tau_\co}$ 
 induces an isomorphism of $\CO_F\otimes\CO_{\shpmuord}$-modules  \[ \gr^1(\uo_{\tau})\simeq \gr^1(\uo_{\tau_\co})^{(p^{j})} = \uo^{(p^{j})}_{\tau_\co}\vert_\shpmuord.\]
\end{enumerate}
\end{lem}
\begin{proof}
The statements follow from the description of the $\mu$-ordinary Dieudonn\'e module and its slope filtration, see \cite[Section 3.1]{EiMa}. In particular, the equality in Equation \eqref{lem2}
follows from the inclusion  $\overline{V}^{j}\vert_{\shpmuord}(U_\tau)\subseteq U_\tauco$.
\end{proof}

We show that, under some restriction on the weights, the \modp reduction of OMOL sheaves extends canonically from $\shpmuord$ to $\Shp$.

The following definition generalizes \cite[Definition 6.3.5]{EiMa} to non-symmetric weights.
  \begin{defi}\label{simple_defi}
We call a positive dominant weight $\kappa$ {\em simple} if it satisfies the following for all $\tau\in\T_F$,  $\co=\co_\tau$: 
\[\text{ if } 0\in\cf(\co_\tau), \text{ then } \kappa_\tau=0;  \text{ and }  \text{ if } 0\not\in\cf(\co_\tau), \text{ then }
\kappa_{i,\tau}=0 \text{ for all } i > \cf(\tau_\co).\]
\end{defi}

\begin{prop}\label{OMOLextend}
For any simple weight $\kappa$,   the restrictions to $\shpmuord$ of the homomorphisms of $\CO_F\otimes\CO_\Shp$-modules $\V_{j,\tau_\co}$, for $\co\in\mathfrak{O}_F$, satisfying $0\not\in\cf(\co)$, and $1\leq j\leq e_\co$, 
 induce an isomorphism of $\CO_F\otimes\CO_{\shpmuord}$-modules  \[ \underline{\uo}^\kappa\simeq \boxtimes_{\tau\in\T_F} \left(\uo^{\kappa_\tau}_{\tau_{\co_\tau}}\right)^{(p^j)}\vert\shpmuord.\]
Under the above identification, we have $\V^\kappa\vert\shpmuord=\varpi^\kappa:\uo^\kappa\to\uuo^\kappa$.
\end{prop}
\begin{proof}
By definition, if $\kappa$ is simple then the associated OMOL sheaf satisfies \[\underline{\uo}^\kappa=\gr^1(\uo)^\kappa,\] 
and  the statement  follows from Lemma \ref{lem}\eqref{lem1}.
\end{proof}
 
In the following, for any simple weight $\kappa$, we use the above identification to extend  from $\shpmuord$ to $\Shp$ the OMOL sheaf $\underline{\uo}^\kappa$, and the canonical projection 
$\varpi^\kappa:\uo^\kappa\to\uuo^\kappa$.

  \begin{remark}\label{hassehere}
Let $\co$ satisfying $0\not\in\cf(\co)$,  and write $e=e_\co$. Then, $\uo_\tauco\vert_\shpmuord=\uuo_\tauco$ and the isomorphism $
\uuo_\tauco\simeq\uo_\tauco^{(p^{e})}\vert_\shpmuord$ naturally extends over $\Shp$ to
the map $
\V_{e_\co} : \uo_\tauco\to \uo_\tauco^{(p^{e_\co})}$. By definition, $\V_{e_\co}$ agrees with the map $\tilde{h}_\tauco$ from Section \ref{vwcase-section}; in particular, its composition with the adjugate  $^*\tilde{h}_\tauco:\uo_\tauco^{(p^{e_\co})}\to\uo_\tauco\otimes |\uo_\tauco|^{p^{e_\co}-1}$ agrees with multiplication by the $\tauco$-Hasse invariant $\HA_\tauco$.
  \end{remark}

  \subsection{Analytic continuation of \modp differential operators on OMOL sheaves}
We prove that the \modp reduction of the differential operators $\underline{\CD}^\lambda$ on OMOL sheaves also extends from $\shpmuord$ to $\Shp$.  

\begin{rmk}
Simple symmetric weight exist if and only if there exists $\co\in\mathfrak{O}$ such that $0,n\not\in\cf(\co)$.
Indeed,
if $\lambda$ is a simple symmetric weight, and $\lambda_\tau\not=0$, then $0,n\not\in\cf(\co_\tau)$.
Also, for any $\ttau\in\Sigma_F$, the basic symmetric weight $\delta(\ttau)$ is simple if and only if $0,n\notin\cf(\co_\ttau)$.
\end{rmk}

\begin{lem}\label{above}
Let $\ttau\in\Sigma_F$. For any simple weight $\kappa$,  there is a differential operator 
\[\uth_\ttau=\uth_{\kappa,\ttau}:\uuo^\kappa\to\uuo^\kappa\otimes\uo^2_\ttau.\]
If $0,n\not\in \cf(\co_\ttau)$ then \[(\mathbb{I}\otimes \varpi_\ttau)\circ \uth_\ttau\vert\shpmuord\equiv  
\underline{\CD}_\ttau\mod p,\]
where $\varpi_\ttau:=\V_{j_\ttau}\otimes\V_{j_{\ttau^*}}:\uo_\ttau^2=\uo_\ttau\otimes\uo_{\ttau^*}\to\uuo^2_\ttau=\uuo_\ttau\otimes\uuo_{\ttau^*}$. 
\end{lem}
\begin{proof}
For any $\co\in\mathfrak{O}_F$ satisfying $0\not\in\cf(\co)$, and $1\leq j\leq e_\co$, let $\tau=\tauco\circ\sigma^j$. We define the operator $\uth_\ttau$ by the Leibniz rule starting from the operators
\[(\uth_\ttau)_\tau=(\underline{D}_\ttau^{(p^j)})_\tauco= (\mathbb{I}\otimes \ks^{-1}_\ttau)\circ \bar{\nabla}^{(p^j)}_\tauco\]\[\uuo_\tau=\uo_\tauco^{(p^j)}\to \uo_\tauco^{(p^j)}\otimes  \Omega^1_{\shp/\overline{\mathbb F}_p}\to \uo_\tauco^{(p^j)}\otimes \uo_\ttau^2=\uuo_\tau\otimes \uo_\ttau^2.\] 

The following commutative diagram show that if $0,n\not\in\cf(\co_\ttau)$ then the operator $\uth_\ttau$ satisfies the congruence $(\mathbb{I}\otimes \varpi_\ttau)\circ \uth_\ttau\vert\shpmuord\equiv
\underline{\CD}_\ttau\mod p.$

\xymatrix{
 \uo_{\tau}\ar[r]_>>>>>>{({D}_\ttau)_\tau}\ar[dd]_\varpi\ar@/^2pc/[rr]^{\CD_\ttau} \ar@/_2pc/[dddd]_{\V_j}&  
 (\hdrone)_{\tau}\otimes\uo^2_\ttau \ar[r]_{\pi_\tau\otimes\mathbb{I}} \ar[dd]_{\varpi\otimes\mathbb{I}}\ar[rdddd]_>>>>>>>>>>>>>>>{\overline{V}^j_\tauco\otimes\mathbb{I}}
 & 
 \uo_{\tau}\otimes\uo^2_\ttau \ar[r]_{\mathbb{I}\otimes\varpi_\ttau}  \ar[dd]^{\varpi\otimes\mathbb{I}} &
  \uo_{\tau}  \otimes\uuo^2_\ttau   \ar[dd]_{\varpi\otimes\mathbb{I}}\ar@/^3pc/[dddd]^{\V_j\otimes
  \mathbb{I}}
  \\
  &&&\\
\gr^1(  \uo_{\tau})\ar[r]_>>>>>{(\underline{D}_\ttau)_\tau}\ar[dd]_{\simeq} \ar@/^2pc/[rrr]^{\underline{\CD}_\ttau}&
  \gr^1(\hdrone)_{\tau}\otimes\uo^2_\ttau\ar@{=}[r]
  & 
  \gr^1(\uo_{\tau})\otimes\uo^2_\ttau \ar[r]_{\mathbb{I}\otimes\varpi_\ttau} \ar[dd]^{ \simeq }& 
  \gr^1(\uo_{\tau})  \otimes\uuo^2_\ttau\ar[dd]_{\simeq } &
  \\
    &&&\\
 \uo_{\tau_\co}^{(p^j)}\ar[r]^{{D}^{(p^j)}_\ttau\quad} \ar@/_2pc/[rr]_{\uth_{ \ttau}}&  
 (\hdrone)_{\tau_\co}^{(p^j)}\otimes\uo^2_\ttau 
 &
 \uo_{\tau_\co}^{(p^j)}\ar@{_(->}[l]
 \otimes\uo^2_\ttau \ar[r]_{\mathbb{I}\otimes\varpi_\ttau}
 & 
 \uo_{\tau_\co}^{(p^j)} 
  \otimes\uuo^2_\ttau
 \\
}
\end{proof}

\begin{thm} \label{OMOLdiffextend}
Let $\Sigma\subseteq \T_F$ satisfying $0\not\in\cf(\co_\tau)$ for all $\tau\in\Sigma$, and let $\lambda$ be a simple symmetric weight supported at $\Sigma$.
Then for any simple weight $\kappa$ supported at $\Sigma$, there is a differential operator 
 \[\uth_\Sigma^\lambda :=\uth_{\kappa}^\lambda: \underline{\uo}^\kappa\to \underline{\uo}^{\kappa+{\lambda} }.\]
satisfying 
\[\uth_\Sigma^\lambda\vert\shpmuord\equiv 
\underline{\CD}^\lambda\mod p.\]

\end{thm}
\begin{proof}

For any simple weight $\kappa$ supported at $\Sigma$, and any $\ttau\in\Sigma_F$ satisfying $0,n\notin\cf(\co_\ttau)$,  consider the differential operators 
\[\uth_{\Sigma,\ttau}=(\mathbb{I}\otimes\varpi_\ttau)\circ \uth_{\kappa,\ttau}:\uuo^\kappa\to\uuo^\kappa\otimes 
\uuo_\ttau^2\]
defined in Lemma \ref{above}.

For any simple symmetric weight $\lambda$ supported at $\Sigma$, we define $\uth_\Sigma^\lambda=\uth^\lambda_\kappa$ by iterating and composing the operators $\uth_{\Sigma,\ttau}$, for $\ttau\in\Sigma_F\cap\Sigma$.
That is, we define
\newcommand{\pr}{{\rm pr}}
\[\uth_\Sigma^\lambda=\pr_{\kappa,{\lambda}}\circ (\mathbb{I}\otimes\pr_{{\lambda}})\circ \uth_{\Sigma,\ttau_1}^{|\lambda_{\ttau_1}|}\circ \cdots\circ \uth_{\Sigma,\ttau_{d}} ^{|\lambda_{\ttau_{d}}|},\]
\[ \uuo^\kappa\to \uuo^\kappa\otimes
\left(\otimes_{\ttau\in\Sigma_F}  
 (\uuo_\ttau^2)^{\otimes|\lambda_\tau|}\right)\to\uuo^\kappa
 \otimes \uuo^{{\lambda}}\to
 \uuo^{\kappa+{\lambda}}
 ,\]
 for any choice of an ordering of the set $\Sigma_F$, $d=[F_0:{\mathbb Q}]$. Lemma \ref{tau-commute} implies that the operator $\uth^\lambda$ is independent of such a choice.  
We observe that, by construction, the operator $\uth_\Sigma^\lambda$ satisfies the given congruence,  as a consequence of the congruences satisfied by the operators $(\mathbb{I}\otimes\varpi_\ttau)\circ\uth_{\ttau}$.
\end{proof}

\subsubsection{The special case of good simple weights}

For weights $\kappa$ which are both good and simple, the OMOL sheaves $\uuo^\kappa$ agree with the restriction of the automorphic sheaves $\uo^\kappa$ over $\shpmuord$, and  for all $\ttau\in\Sigma_F$, the operators $\Theta_{\kappa, \ttau}$ and $\uth_{\kappa,\ttau}$ both extend the mod $p$ reduction of the basic Maass--Shimura differential operators $\CD_{\kappa,\ttau}$ on $\uo^\kappa$. 
We compare the two constructions.

Set  $\Upsilon=\{\tau_\co\in\co| \co\in\mathfrak{O} \text{ satisfying } 0\not\in\cf(\co)\}$, which we regard also as $\Upsilon\subseteq \mathfrak{O}$.

\begin{rmk}\label{rmkabove}
If $\kappa$ is a weight supported at $\Upsilon$, 
then $\kappa$ is good and simple. 
Vice versa, if a weight $\kappa$ is good and simple, then $\kappa$ is supported at $\{\tau\in\T_F| \,0\not\in\cf(\co_\tau) \text{ and } \cf(\tau)=\cf(\tau_\co)\}$. 
\end{rmk}

For $\kappa$ any weight supported at $\Upsilon$, we write
${^*\varpi}^\kappa$ for  the adjugate of $\varpi^\kappa: \uo^\kappa\to\uuo^\kappa$, and we have
\[\HA^{r(\kappa)-1}=\prod_{\tauco\in\Upsilon} \HA_{\tauco}^{\max(r(\kappa_\tauco)-1,0)}.\] We refer to Definition \ref{determinantpower} for the notation $r(\kappa)\in\ZZ^{|\T|}$.
\begin{prop}\label{compare_prop}
Let $\Upsilon_0\subseteq\Upsilon$, and $\kappa$ be a weight supported at $\Upsilon_0$. 
For any $\ttau\in\Sigma_F$, 
\[\HA^{r(\kappa)-1}\cdot\Theta_{\Upsilon_0, \ttau}=({^*\varpi}^\kappa\otimes \mathbb{I})\circ \uth_{\ttau}\circ \varpi^\kappa.\]
Furthermore, if $\ttau\in\Sigma_F$ satisfies $0,n\not\in\cf(\co_\ttau)$, then
\[\HA^{r(\kappa)-1}\cdot (\mathbb{I}\otimes \varpi_{\ttau})\circ \Theta_{\Upsilon_0, \ttau}=({^*\varpi}^\kappa\otimes \mathbb{I})\circ \uth_{\Upsilon_0,\ttau}\circ \varpi^\kappa.
\]
\end{prop}
\begin{proof} Let $\ttau\in\Sigma_F$.
 For each $\tauco\in \Upsilon$,  
 the commutativity of the diagram in the proof of Lemma \ref{above} implies \[\Theta_{\tauco, \ttau}=({^*\varpi}_\tauco\otimes \mathbb{I})\circ \uth_{\ttau}\circ \varpi_\tauco,\]
\[\omega_\tauco\to 
\underline{\uo}_\tauco=\uo_\tauco^{(p^e)}\to
\uo_\tauco^{(p^e)}\otimes \omega^2_\ttau\to
\uo_\tauco\otimes|\uo_\tauco|^{p^e-1}\otimes \uo_\ttau^2,\]
where by definition $\varpi_\tauco= 
\tilde{h}_\tauco$, and  ${^*\tilde{h}}_\tauco\circ {\tilde{h}}_\tauco$ is multiplication by $\HA_\tauco$ (see Remark \ref{hassehere}).
Furthermore, if $\ttau\in\Sigma_F$ satisfies $0,n\not\in\cf(\co_\ttau)$, composition with the map $(\mathbb{I}\otimes \varpi_{\ttau})$ yields
\[(\mathbb{I}\otimes \varpi_{\ttau})\circ \Theta_{\tauco, \ttau}=({^*\varpi}_\tauco\otimes \mathbb{I})\circ \uth_{\tauco, \ttau}\circ \varpi_\tauco.\]

For $\kappa$ any weight supported as $\Upsilon_0$,  we deduce the statement from the above equalities
by comparing the construction of the operators $\Theta_{\Upsilon_0,\ttau}$ and $\uth_\ttau, \uth_{\Upsilon_0,\ttau}$ on $\uo^\kappa$, 
and observing that ${\varpi}^\kappa= {\tilde{h}}^\kappa$ and
 ${^*\tilde{h}}^\kappa\circ {\tilde{h}}^\kappa$ is multiplication by $\HA^{r(\kappa)}$.
 \end{proof}

\subsection{A new class of entire \modp differential operators}\label{last_sec}
 We conclude by introducing a new class of weight-raising \modp differential operators $\tth^\lambda$ on \modp automorphic forms. 
These operators are obtained by iterating and composing basic differential operators $\tth_\ttau$ which are defined by composing the operators $\Theta_\ttau$ with the projections $\varpi_\ttau:\omega^2_\ttau\to\uuo_\ttau^2$, for any $\ttau\in\Sigma_F$ satisfying $0,n\not\in\cf(\co_\ttau)$.  In order to iterate and compose the operators $\tth_\ttau$, we observe that the OMOL sheaves $\uuo_\ttau^2$ also arise as subsheaves of automorphic sheaves of higher good weights.

For each $\co \in\mathfrak{O}$  satisfying  $0,n\not\in\cf(\co)$, we choose $\tauco\in\co$ such that $\cf(\tauco)=\min\cf(\co)$, and 
set  \begin{equation}\label{upsilon_def}
\Upsilon=\{\tau_\co| \co\in\mathfrak{O} \text{ satisfies } 0,n\not\in\cf(\co)\}.\end{equation} 

\begin{defi}\label{twistweight_defi}
For any simple symmetric weight $\lambda$, $\lambda=(\lambda_\tau)_{\tau\in\T_F}$, we define the {\em $\Upsilon$-twist } $\tlambda=\tlambda^\Upsilon$ of $\lambda$ by
\[\tlambda_\tau=0 \text{ if } \tau\not\in\Upsilon, \text{ and }\tlambda_{\tau}=\sum_0^{e_\co-1} p^j\lambda_{\tauco\circ\sigma^j} \text{ if } \tau=\tauco\in\Upsilon.\]
By definition, the weight $\tlambda^\Upsilon $ is supported at $\Upsilon$, and hence it is good and simple. Note that $\tlambda^\Upsilon $ is not symmetric.
We write $\tilde{\delta}(\ttau)$ for the $\Upsilon$-twist of $\delta(\ttau)$, 
for $\ttau\in\Sigma_F$.
\end{defi}

For any $\CO_\Shp$-module $\mathcal{F}$, we write \[f:\mathcal{F}^{(p)}\to{\rm Sym}^{p}\mathcal{F}\] for the natural morphism of $\CO_\Shp$-module.  By abuse of notation, we also write $f$ for the induced morphisms $\mathcal{F}^{(p^j)}\to{\rm Sym}^{p^j}\mathcal{F}$, for  $j\geq 1$.

For any $\ttau\in\Sigma_F$ satisfying $ 0,n\not\in\cf(\co_\ttau)$,
 let 
\[p_\ttau=f\circ \varpi_\ttau:\uo^2_\ttau=\uo_\ttau\otimes\uo_{\ttau^*}\to\uuo_\ttau^2=\uo^{(p^{j})}_\tauco\otimes\uo^{(p^{j^*})}_{\tau_{\co^*}} \to  
\Sym^{p^{j}}(\uo_\tauco)\otimes \Sym^{p^{j^*}}(\uo_{\tau_{\co^*}}),\]
where $\co=\co_\ttau$. For $\ttau=\tauco$ (resp. $\ttau=\tau_{\co^*}$), set $j=0$ (resp. $j^*=0$).

If $\co\not=\co^*$, the sheaf $\Sym^{p^{j}}(\uo_\tauco)\otimes  
\Sym^{p^{j^*}}(\uo_{\tau_{\co^*}})$ is an automorphic sheaf, and its weight is $\tilde{\delta}(\ttau)$.
If $\co=\co^*$, by abuse of notation we still denote by $p_\ttau$ its composition with the natural morphism \[\Sym^{p^{j}}(\uo_\tauco)\otimes  
\Sym^{p^{j^*}}(\uo_{\tau_{\co^*}})=\Sym^{p^{j}}(\uo_\tauco)\otimes  
\Sym^{p^{e/2+j}}(\uo_{\tau_{\co}})\to\Sym^{p^j(1+p^{e/2})} (\uo_\tauco);\] the sheaf $\Sym^{p^j(1+p^{e/2})} (\uo_\tauco)$ is an automorphic sheaf, and its weight  is
${\tdelta(\ttau)}$.

\begin{lem}\label{twistpr}
For any simple symmetric weight $\lambda$,  ${\rm pr}_{\tlambda}: \otimes_{\tauco\in\Upsilon} (\uo_\tauco)^{\otimes |\tlambda_\tauco|}\to\uo^{\tlambda}$ factors via the homomorphism
$ \otimes_{\tauco\in\Upsilon} (\uo_\tauco)^{\otimes |\tlambda_\tauco|}\to\otimes_{\ttau\in\Sigma_F} (\uo^{\tdelta(\ttau)})^{\otimes |\lambda_\ttau|}$.
\end{lem}
\begin{proof}
It suffices to observe that, for any simple symmetric weight $\lambda$, the morphisms $p_\ttau:\uo^2_\ttau\to\uuo_\ttau^2\to\uo^{\tdelta(\ttau)}$ induce (via Schur functors) a morphism $p^\lambda:\uo^\lambda\to \uuo^\lambda\to\uo^{\tlambda}$, which fits in the following commutative diagram
\[\xymatrix{
& \bigotimes_{\tau}(\uo_\tau)^{\otimes |\lambda_\tau|} \ar@/^2pc/[rr]^{{\rm pr}_\lambda}\ar[r]\ar[d]_{\otimes_\tau (\V_\tau)^{\otimes |\lambda_\tau|}}
& \bigotimes_{\ttau}(\uo^2_\ttau)^{\otimes |\lambda_\ttau|} \ar[r] \ar[d]_{\otimes_\ttau (p_\ttau)^{\otimes |\lambda_\ttau|}}
&\uo^\lambda  \ar[d]_{p^\lambda}\\
\bigotimes_{\tauco} (\uo_\tauco) ^{\otimes |\tlambda_\tauco|} \ar[r]\ar@/_2pc/[rrr]_{{\rm pr}_{\tlambda}}
&\bigotimes_{\tauco}\bigotimes_{\tau\in\co}\left({\rm Sym}^{p^{j_\tau}}(\uo_\tauco)\right)^{\otimes |\lambda_\tau|} \ar[r]
&
\bigotimes_{\ttau}\left(\uo_\ttau^{\tdelta(\ttau)}\right)^{\otimes |\lambda_\tau|}\ar[r] &\uo^{\tlambda}
}\]
\end{proof}

Let $\ttau\in\Sigma_F$ satisfy $ 0,n\not\in\cf(\co_\ttau)$. 
For any $\Sigma\subseteq\T_F$, and any good weight $\kappa$ supported at $\Sigma$, we define a differential operator $\tth_{\Sigma, \ttau}=\tth_{\Upsilon, \Sigma,\ttau}$, on \modp automorphic forms of  weight 
$\kappa$,  by
\[\tth_{\Sigma, \ttau}={\rm pr}_{\kappa+\kappa_{\ha,\Sigma},\delta_{\Upsilon}(\ttau)}\circ (\mathbb{I}\otimes p_\ttau)\circ \Theta_{\Sigma,\ttau},\]
\[\uo^\kappa\to\uo^{\kappa +\kappa_{\ha,\Sigma}} \otimes \uo^2_\ttau\to 
\uo^{\kappa+\kappa_{\ha,\Sigma}}  \otimes  \uo^{\tdelta(\ttau)}\to\uo^{\kappa+{\kappa_{\ha,\Sigma}} +\tdelta(\ttau)}.\]

If $\Upsilon\subseteq \Sigma$
, then $\kappa+{\kappa_{\ha,\Sigma}} +\tdelta^\Upsilon(\ttau)$ is also a good weight supported at $\Sigma$. Hence, the operators $\tth_{\Sigma,\ttau}$ can be iterated and composed without restrictions.
For any simple symmetric weight $\lambda$,  and any choice of an ordering of the set $\Sigma_F$, $d=[F_0:{\mathbb Q}]$, we define 
\[\tth^\lambda_{\Sigma}:={\rm pr}_{\kappa+\kappa_{\ha,\Sigma},\tlambda}\circ (\mathbb{I}\otimes{\rm pr}_{\tlambda})\circ 
\tth_{\Sigma,\ttau_1}^{|\lambda_{\ttau_1}|}\circ \cdots \circ \tth_{\Sigma,\ttau_d}^{|\lambda_{\ttau_d}|},\]
\[\uo^\kappa\to\uo^{\kappa +(|\lambda|/2)\kappa_{\ha,\Sigma}} \otimes \left(\otimes_{\ttau\in\Sigma_F} (\uuo^2_\ttau)^{|\lambda_\ttau|}\right)\to 
\uo^{\kappa+(|\lambda|/2)\kappa_{\ha,\Sigma}}  \otimes \uuo^{\lambda}\to
\uo^{\kappa+(|\lambda|/2)\kappa_{\ha,\Sigma}}  \otimes \uo^{\tlambda}\to
\uo^{\kappa+{(|\lambda|/2)\kappa_{\ha,\Sigma}} +\tlambda}.\]
 
By Lemmas \ref{tau-commute} and \ref{twistpr}, the operators $\tth_\Sigma^\lambda$ are well defined, independent of the choice of an ordering of the set $\Sigma_F$.
We deduce the following result, concerning entire theta operators that raise the weights by weights that are not symmetric, and which do not arise as the $\mod p$ reductions of $p$-adic Maass--Shimura operators.

\begin{thm}\label{newops-thm}
Let  $\Upsilon$ as in Equation \eqref{upsilon_def}, and assume $\Upsilon\not=\emptyset$.
 For  any simple symmetric weight $\lambda$, and  any $\Upsilon\subseteq \Sigma\subseteq \T_F$,  there is a differential operator on \modp automorphic forms of weight $\kappa$,  for $\kappa$ any good weight supported at $\Sigma$, 
 \[\tth^\lambda_{\Sigma}=\tth^\lambda_{\Upsilon, \Sigma}:\uo^\kappa\to \uo^{\kappa+(|\lambda|/2)\kappa_{\ha,\Sigma}+\tlambda} ,\] which raises the weight $\kappa$ by $(|\lambda|/2)\kappa_{\ha,\Sigma}+\tlambda^\Upsilon$, where $\tlambda^\Upsilon$ is as  in Definition \ref{twistweight_defi}.
\end{thm}

\begin{rmk}\label{DSG_compare}
For $F$ quadratic imaginary, $p$ inert in $F$, and indefinite signature, the choice of $\Upsilon, \ttau$ as above is unique, and the associated operator $\tth_{\Upsilon,\ttau}$ agrees 
(up to multiplication by the Hasse invariant $\HA_\Upsilon$) with the operator $\Theta$   constructed in \cite[Section 4]{DSG2}. In {\em loc. cit.}, the operator $\Theta$ is defined on  automorphic forms of scalar weights supported at $\Upsilon$, and can be iterated when the signature of the unitary group is $(n,1)$.
\end{rmk}

\section{Application to mod $p$ Galois representations}\label{GalOne-section}
In this section, we apply the results from the previous sections (more precisely, Section \ref{modpdiffop_sec} on analytic continuation of differential operators, and Section \ref{last_sec} on a new class of entire differential operators) to Galois representations.  For the first of these classes of differential operators, the results in this section remove the splitting constraint on $p$ from the analogous results in \cite[Sections 4 and 5]{EFGMM}.

\subsection{Commutation relations with Hecke operators}

Following the same approach as in \cite[Section 4]{EFGMM}, we study the commutation relations with Hecke operators (and Hida's $\mu$-ordinary projectors built from Hecke operators at $p$), of the \modp differential operators $\Theta_\tau$ and $\Theta^\lambda$ constructed in Section \ref{modpdiffop_sec}, and of the \modp differential operators  $\tth_\tau$ and $\tth^\lambda$ constructed in Section \ref{last_sec} (resp. of the $p$-adic Maass--Shimura differential operators $\CD^\lambda$ constructed in \cite{EiMa}). 

\begin{rmk} The definition of the differential operators $\Theta^\lambda=\Theta^\lambda_\Sigma$ and $\tth^\lambda=\tth^\lambda_\Sigma$ depends on the choice of a non-empty set $\Sigma\subseteq \T_F$. 
As we shall see the results in this section do not depend on $\Sigma\subseteq \T_F$. We therefore drop the subscript $\Sigma$ from our notation.  
\end{rmk}

\begin{rmk}
The definition of the differential operators $\tth^\lambda=\tth_{\Upsilon, \Sigma}^\lambda$ depends on the existence and choice of a non-empty set $\Upsilon$
as in Equation (\ref{upsilon_def}). In the following, we assume there exists $\Upsilon$ nonempty, we fix such a choice and drop the subscript $\Upsilon$ from our notation.
\end{rmk}

\begin{rmk}
In our discussion Hecke operators below, following the approach of \cite{FaltingsChai}, we only use the fact that the Hecke action is formulated in terms of algebraic correspondences, so other approaches similarly formulated in terms of algebraic correspondences (even if they are normalized differently) also fit into this framework and, in particular, other normalizations would not affect the statements of Corollaries \ref{Heig} and \ref{Galcor}.  (One reason for making this observation is that when writing double coset representations of Hecke operators, one sometimes needs to normalize them to work integrally, as explained in moving from the ``naive,'' ``unnormalized'' Hecke operators expressed in terms of double cosets to normalized, integral Hecke operators in \cite[\S 1]{FaPi}.)
\end{rmk}

\subsubsection{Hecke operators away from $p$}
We briefly recall the definition of the action of the prime-to-$p$ Hecke operators on \modp (resp. $p$-adic) automorphic forms. (We refer to \cite[Ch. VII, \S3]{FaltingsChai}, and also \cite[Section 4.2]{EFGMM}, for details.) 

Fix a rational prime $\ell$, $\ell\neq p$. {\bf We assume that $\ell$ is a prime of good reduction for $\Sh_\compact$}. I.e., for each prime $v|\ell$ of $F$, we assume that $\compact_v$ is a hyperspecial maximal compact subgroup of $G_v=G(F_v)$.

Let $\ell-{\rm Isog}$ denote the moduli space of $\ell$-isogenies over $\Sh$.  
We denote by $\varphi: \projection_1^\ast\Auniv\rightarrow\projection_2^\ast\Auniv$ the universal $\ell$-isogeny, where $\projection=(\projection_1, \projection_2):\ell-{\rm Isog}\rightarrow \Sh\times \Sh$ denotes the natural structure morphism. 
Similarly to \cite[Definition 4.2.2]{EFGMM}, for any connected component $Z$ of $\ell-{\rm Isog}$, 
we denote by $T_{(Z,\varphi)}$ the natural action of $(Z, \varphi)$ on $H^0(\Sh, \omega^\kappa)$
via algebraic correspondence.  By abuse of notation, we also denote by $T_{(Z,\varphi)}$ the induced actions on \modp automorphic forms over $\Shp$ (resp. on $p$-adic automorphic forms over $\shmuord$).
We generalize \cite[Theorem 4.2.4]{EFGMM}.

\begin{prop}\label{ellhecke_thm}
Let $(Z, \varphi)$ be a connected component  of the moduli space of $\ell$-isogenies over $\Sh$, with $\nu(\varphi)$  the similitude factor of $\varphi$. Let  $\kappa,\lambda$ be two weights, and assume $ \lambda $ symmetric.
\begin{enumerate}
 \item 
 $T_{(Z,\varphi)} \circ \CD^\lambda=
\nu(\varphi)^{|\lambda|/2}
\CD^\lambda\circ
 T_{(Z,\varphi)} $.\label{part1ee}
\item   
$T_{(Z,\varphi)} \circ \Theta^\lambda=
\nu(\varphi)^{|\lambda|/2}
\Theta^\lambda\circ
 T_{(Z,\varphi)}  $   if  both $\kappa$ and  either $\lambda$ or $\lambda-\delta(\tau)$, for some $\tau\in\Sigma_{F}$,  are good.\label{part2ee}
\item $T_{(Z,\varphi)} \circ \tth^\lambda=
\nu(\varphi)^{|\lambda|/2}
\tth^\lambda\circ
 T_{(Z,\varphi)} , $  if $\kappa$ is good and $\lambda$ is simple. \label{part3ee}
 \end{enumerate}
\end{prop}
\begin{proof}
For Part \eqref{part1ee}, by construction of the operator $\CD^\lambda$, the statement reduces to the special cases \[T_{(Z,\varphi)} \circ \CD_{\tau}=
\nu(\varphi)
\CD_{\tau}\circ
 T_{(Z,\varphi)},\]
for any $\tau\in\Sigma_{F}$. By the definition, $\CD_{\tau}=(\pi_{\tau}\otimes\ks^{-1}) \circ \nabla_{\CA/\shmuord}$, and the commutation relations  follow from the functoriality of the Gauss--Manin connection, the definition of the morphisms $\pi_{\tau}$ and the equality
$\nu(\varphi)\KS=\KS\circ(\varphi^*\otimes\varphi^*)$.

For Part \eqref{part2ee}, for any $\tau\in\Sigma_{F}$,  the operators 
$\Theta_{\tau}$ are defined as $\Theta_{\tau}=(\Pi_{\tau}\otimes\ks^{-1}) \circ \nabla_{\CA/\shp}$, and the same argument  yields the result for the weight $\lambda=\delta(\tau)$.
On the other hand, for a more general weight $\lambda$, the operators $\Theta^\lambda$ are not constructed by composition/iteration of the operators
$\Theta_{\tau}$, thus the statement does not reduce to the aforementioned case. When $\lambda$ is a good weight, the same argument still applies, with minor changes.
When $\lambda-\delta(\tau)$, for some $\tau\in\Sigma_{F}$,  is a good weight, then the statement follows from the equality $\Theta^{\lambda+\delta({\tau})}=\Theta_{\tau}\circ\Theta^\lambda$, and the previously established cases.

For Part \eqref{part3ee}, by construction of the operator $\tth^\lambda$, the statement reduces to the special cases 
\[T_{(Z,\varphi)}\circ \tth_\tau = \nu(\varphi)\cdot \tth_\tau\circ T_{(Z,\varphi)}\]
for
any  $\tau\in\Sigma_F$ satisfying  
$0,n\not\in\cf(\co_\tau)$.
By definition, $\tth_\tau=(\mathbb{I}\otimes p_\tau)\circ\Theta_\tau$,
and the statement follow from Part (2) and the functoriality of the morphisms $p_\tau$.
\end{proof}

Finally, we recall the action of the prime-to-$p$ Hecke operators. 
We define ${\mathcal {H}}_0(G_\ell,\IQ)$ to be the $\IQ$-subalgebra of   the local Hecke algebra ${\mathcal{H}}(G_\ell,\compact_\ell;\IQ)$ generated by locally constant function supported on cosets $\compact_\ell\gamma\compact_\ell$, for $\gamma\in G_\ell$ an integral matrix. 
Then, the action on \modp (resp. $p$-adic) automorphic forms of the  prime-to-$p$ Hecke operators  agrees with that of the prime-to-$p$ algebraic correspondences, via pullback under the map of $\IQ$-algebras
\[h_\ell: {\mathcal {H}}_0(G_\ell,\IQ)\to \IQ[\ell-{\rm Isog}/Y]\]
where $Y=\shp/\kappa(\mathfrak{p})$ (resp.  
$\shpmuord/\Witt$), 
which to  any double coset $\compact_\ell\gamma\compact_\ell$, with $\gamma$ an integral matrix in $G_\ell$, associates  the union of those connected component of $\ell-{\rm Isog}$ where the universal isogeny is an $\ell$-isogeny of type $\compact_\ell\gamma\compact_\ell$.

The following Corollary is an immediate consequence of Proposition \ref{ellhecke_thm}.

\begin{coro}\label{Heig}
Let $f$ be a \modp Hecke eigenform of weight $\kappa$ on $\Shp$. 
Assume $\kappa$  is good.  Then:
\begin{enumerate}
\item 
For any symmetric weight $\lambda$, such that either $\lambda$ or $\lambda-\delta(\tau)$ is good, for some $\tau\in \Sigma_{F}$, if 
$\Theta^\lambda(f)$ is nonzero then it is a \modp Hecke eigenform.

\item 
For any simple symmetric weight $\lambda$, if $\tth^\lambda(f)$ is nonzero, then $\tth^\lambda(f)$ is a \modp Hecke eigenform. 
\end{enumerate}
\end{coro}

\subsubsection{Hecke operators at $p$}

We briefly recall the definition of the action of Hecke operators  at $p$, on $p$-adic automorphic forms over the $\mu$-ordinary locus $\shmuord$. (We refer to \cite[Ch. VII, \S4]{FaltingsChai}, and also \cite[Section 4.3]{EFGMM}, for details.) 

Let $p-{\rm Isog}^o$  denote the moduli space of $p$-isogenies over the $\mu$-ordinary locus $\shmuord$. 
For  any  connected component $(Z, \varphi)$ of $p-{\rm Isog}^o$, we write $T_{(Z,\varphi)}$ for the  action $(Z,\varphi)$ on $p$-adic automorphic forms over $\shmuord$.

We generalize \cite[Theorem4.3.3]{EFGMM}. 

\begin{prop}
For any connected component $(Z, \varphi)$ of $p-{\rm Isog}^o$, with $\nu(\varphi)$  the similitude factor of $\varphi$, and any two  weights $\kappa,\lambda$, with $\lambda $ symmetric, 
\[T_{(Z,\varphi)} \circ \CD^\lambda=
\nu(\varphi)^{|\lambda|/2}
\CD^\lambda\circ
 T_{(Z,\varphi)} .\]
 In particular, if $\nu(\varphi)>0$ then $T_{(Z,\varphi)} \circ \CD^\lambda=0$.
\end{prop}

For any connected component $(Z, \varphi)$ of $p-{\rm Isog}^o/\shmuord\otimes_\Witt {\mathbb F}$, 
we define 
 the {\em normalized action} of $(Z,\varphi)$ on \modp automorphic forms over $\shpmuord$ as  $t_{(Z,\varphi)}:=\mu^{-1}(Z,\varphi)T_{(Z,\varphi)}$, where $\mu(Z,\varphi)$ is the purely inseparable multiplicity of the geometric fibers of $Z\to\shpmuord$.

We are now ready to introduce the action of the Hecke operators at $p$.
Following {\it loc. cit.}, we identify $\Levi\times \mathbb{G}_m$ with the appropriate maximal Levi subgroup $M$ of $\mathcal{G}$ over $\CO_{\reflexfield_\mathfrak{p}}$, and 
realize the local Hecke algebra ${\mathcal{H}}(M(\reflexfield_{\mathfrak{p}}), M(\CO_{\reflexfield_\mathfrak{p}});\IQ)$ as a subalgebra of ${\mathcal{H}}(G(\reflexfield_{\mathfrak{p}}), {\mathcal G}(\CO_{\reflexfield_\mathfrak{p}});\IQ)$. (Note that, when the ordinary locus is nonempty, the Levi subgroup $M$ is defined over $\ZZ_p$.)
We set $M_p:=M(\reflexfield_\mathfrak{p})$, and define ${\mathcal {H}}_0(M_p,\IQ)$ to be the $\IQ$-subalgebra of  the local Hecke algebra ${\mathcal{H}}(M_p, M(\CO_{\reflexfield_\mathfrak{p}});\IQ)$ generated by locally constant function supported on cosets $M(\CO_{\reflexfield_\mathfrak{p}})\gamma M(\CO_{\reflexfield_\mathfrak{p}})$, for $\gamma\in M_p$ an integral matrix.

Then, the action of the  Hecke operators at $p$  on \modp automorphic forms over $\shmuord$  agrees with the normalized action of the $p$-power algebraic correspondences, via pullback under the map of $\IQ$-algebras
\[h_p: {\mathcal {H}}_0(M_p,\IQ)\to \IQ[p-{\rm Isog}^o/\shpmuord],\]
which to  any double coset $\compact_\ell\gamma\compact_\ell$, with $\gamma$ an integral matrix in $G_\ell$, associates  the union of those connected component of $p-{\rm Isog}^o$ where the universal isogeny is a $p$-isogeny of type $M(\CO_{\reflexfield_\mathfrak{p}})\gamma M(\CO_{\reflexfield_\mathfrak{p}})$.

\subsubsection{Ordinary projector} When the ordinary locus is nonempty,  in \cite[Section 4.3.1]{EFGMM} we also address the interaction between differential operators 
and Hida's ordinary projector.  More generally, even when the ordinary locus is empty, we have the $\mu$-ordinary project $\be$, which coincides with Hida's ordinary projector when the ordinary locus is nonempty.  The $\mu$-ordinary projector $\be$ was introduced in a general setting in \cite[Section 6.2]{HidaAsian} and later explored in the context of $p$-adic automorphic forms over the $\mu$-ordinary locus of unitary Shimura varieties in \cite{brasca-rosso}.  The Hecke operators at $p$ are defined in \cite[Section 3]{brasca-rosso}, and then $\be$ is built from them analogously to in the ordinary case.
By a similar argument to the proof of \cite[Corollary 4.3.5]{EFGMM}, we then have Corollary \ref{ordprojector-cor}, which specializes to \cite[Corollary 4.3.5]{EFGMM} when the ordinary locus is nonempty.

\begin{cor}\label{ordprojector-cor}
For any  weight $\kappa$ and symmetric weight $\lambda$,  $\be\diffop_\kappa^\lambda = 0.$
\end{cor}

\subsection{Consequences for Galois representations}
Let $\chi$ denote the  \modp cyclotomic character. Recall  $\hat{\nu}:\Gm\to\hat{G}$ is the cocharacter dual to the similitude factor $\nu:G\to\Gm$.

By a similar argument to the proof of  \cite[Theorem B]{EFGMM}, we  extend  \cite[Theorem B]{EFGMM} to our context.

\begin{cor} [Action of differential operators on \modp Galois representations]\label{Galcor}
Let $f$ be a
\modp Hecke eigenform  on $\shp$ of weight $\kappa$, for $\kappa$  a  weight supported at $\Sigma$, for some $\Sigma\subseteq \CT$, and $\rho:\Gal(\bar{F}/\F)\to \hat{G}(\mathbb{F)}$ a continuous representation.

Assume $\kappa$ is good. Let $\lambda$ be a symmetric weight.
\begin{enumerate}
\item Suppose either  $\lambda-\delta({\tau})$, for some $\tau\in\Sigma_F$, or $\lambda$ is good; set $\lambda'=\lambda-\delta({\tau})$ or $\lambda'=\lambda$, respectively.
Assume $\Theta_\Sigma^\lambda (f)$  is nonzero.

Then, the Frobenius eigenvalues of $\rho$ agree with the Hecke eigenvalues of the form $f$ (as defined in Conjecture \ref{gal-rep-conj}) if and only if the Frobenius eigenvalues of 
$(\hat{\nu}^{|\lambda|/2}\circ\chi)\otimes \rho$ agree with the Hecke eigenvalues of the form $\Theta_\Sigma^\lambda (f)$. 

In particular, if $\rho$ is modular of weight $\kappa$, then $(\hat{\nu}^{|\lambda|/2}\circ\chi)\otimes \rho$ is modular of weight  $\kappa + \lambda + (|\lambda|/2)\kappa_{\ha,\Sigma}+||\lambda'||\kappa_\ha$.

\item Suppose $\lambda$ is simple; fix $\Upsilon$ as in Equation (\ref{upsilon_def}).
Assume $\Upsilon\subseteq\Sigma$ and $\tth_{\Upsilon,\Sigma}^\lambda (f)$  is nonzero.

Then, the Frobenius eigenvalues of $\rho$ agree with the Hecke eigenvalues of the form $f$ (as defined in Conjecture \ref{gal-rep-conj}) if and only if the Frobenius eigenvalues of 
$(\hat{\nu}^{|\lambda|/2}\circ\chi)\otimes \rho$ agree with the Hecke eigenvalues of the form $\tth_{\Upsilon,\Sigma}^\lambda (f)$. 

In particular, if $\rho$ is modular of weight $\kappa$, then $(\hat{\nu}^{|\lambda|/2}\circ\chi)\otimes \rho$ is modular of weight $\kappa + \tlambda^\Upsilon + (|\lambda|/2)\kappa_{\ha,\Sigma}$.
\end{enumerate}
\end{cor}

As in \cite[Section 5.2]{EFGMM}, 
the above result is a first step in the use of $\Theta$-operators to investigate Serre's weight conjecture (as, e.g.,  in the specific case of $\GSp_4(\IQ)$ in \cite[Theorems 1.1 and 1.2]{Yama}) on minimal weights of modularity for \modp Galois representations, or more generally how the weights of modularity vary under twists by the cyclotomic character.  Some preliminary results on $\Theta$-cycles analogous to \cite[Theorem on p. 55]{Katz-modular} and \cite[Section 5.2]{EFGMM} also hold in this context, when restricting to scalar weights.  As in \cite{Yama}, the general case, beyond scalar weights, is much more subtle.

As first observed for a special case in \cite[Section 4.1]{DSG} (and also in \cite[Section 5]{DSG2}),
the cycles described by the modular weights under the action of the operators $\tth$ are substantially different from those obtained under the action of the operators $\Theta$, which is likely to provide an advantage in the study of Serre's weight conjecture.

\bibliography{muordinaryAnalyticContinuationbib}
\printindex
\end{document}